\documentclass[a4paper,11pt]{amsart}
\usepackage[english]{babel} 
\usepackage{float} 
\usepackage{amsfonts} 
\usepackage{amssymb} 
\usepackage{amsmath} 
\usepackage{amsthm} 
\usepackage[dvips]{graphicx} 
\usepackage[dvips]{rotating} 
\usepackage{curves} 
\usepackage{xcolor} 
\usepackage{eepic} 
\usepackage{enumitem} 
\usepackage{textpos} 
\usepackage{fancyvrb} 
\usepackage{breqn} 
\usepackage{csquotes} 
\usepackage{tikz} 
\usetikzlibrary{patterns} 
\usepackage{tabstackengine} 
\usepackage{tabularx} 
\stackMath 
\usepackage{algorithm,algorithmic}
\usepackage{blkarray}
\usepackage{mathtools}
\usepackage{lipsum}
\usepackage{longtable}
\usepackage{amsaddr}
\usepackage{makecell}

\RequirePackage{color}
\RequirePackage[colorlinks, urlcolor=my-blue,linkcolor=my-red,citecolor=my-green]{hyperref}
\definecolor{my-blue}{rgb}{0.0,0.0,0.6}
\definecolor{my-red}{rgb}{0.5,0.0,0.0}
\definecolor{my-green}{rgb}{0.0,0.5,0.0}
\definecolor{nicos-red}{rgb}{0.75,0.0,0.0}
\definecolor{light-gray}{gray}{0.6}
\definecolor{really-light-gray}{gray}{0.8}
\definecolor{sussexg}{rgb}{0.0,0.5,0.5}
\definecolor{sussexp}{rgb}{0.5,0.0,0.5}

\definecolor{darkgreen}{rgb}{0.0,0.5,0.0}
\definecolor{darkblue}{rgb}{0.0,0.0,0.3}
\definecolor{nicosred}{rgb}{0.65,0.1,0.1}
\definecolor{light-gray}{gray}{0.7}
\allowdisplaybreaks[1]

\newtheorem{theorem}{Theorem}[section]
\newtheorem{lemma}[theorem]{Lemma}

\newtheorem{cor}[theorem]{Corollary}

\newtheorem{assumps}[theorem]{Assumptions}

{\theoremstyle{definition}
\newtheorem{defn}[theorem]{Definition}
\newtheorem{example}[theorem]{Example}

\newtheorem{rem}[theorem]{Remark}

}

\newcommand{\cP}{{\mathcal{P}}}

\newcommand{\bR}{{\mathbb{R}}}
\newcommand{\bZ}{{\mathbb{Z}}}
\newcommand{\bN}{{\mathbb{N}}}
\newcommand{\OP}{\Omega^{P}}
\newcommand{\OPtt}{\Omega^{P^{\theta}}}
\newcommand{\ymtt}{\mathbf{y}^{m,\theta}}

\newcommand{\Mtt}{M^{\theta}}

\newcommand{\Mee}{M^{\eta}}
\newcommand{\bjmtt}{b_{j}^{m,\theta}}

\newcommand{\Rmtt}{R^{m,\theta}}
\newcommand{\lljmtt}{\lambda_{j}^{m,\theta}}

\newcommand{\llssmtt}{\lambda_{\sigma}^{m,\theta}}
\newcommand{\aassjmtt}{\alpha_{\sigma,j}^{m,\theta}}
\newcommand{\bbjmtt}{\beta_{j}^{m,\theta}}
\newcommand{\bbilee}{\beta_{i}^{\ell,\eta}}
\newcommand{\rrmtt}{\rho^{m,\theta}}
\newcommand{\Ptt}{P^{\theta}}
\newcommand{\cPtt}{{\mathcal{P}}^{\theta}}
\newcommand{\pjtt}{p_{j}^{\theta}}
\newcommand{\rjmtt}{r_{j}^{m,\theta}}
\newcommand{\Pmtt}{P^{m,\theta}}
\newcommand{\pjmtt}{p_{j}^{m,\theta}}
\newcommand{\aijeettlm}{a_{i,j}^{\eta,\theta;\ell,m}}

\newcommand{\ylee}{\mathbf{y}^{\ell,\eta}}
\newcommand{\pilee}{p_{i}^{\ell,\eta}}
\newcommand{\llilee}{\lambda_{i}^{\ell,\eta}}
\newcommand{\rrlee}{\rho^{\ell,\eta}}

\newcommand{\xxi}{\mathbf{x}_{i}}
\newcommand{\xxiln}{\mathbf{x}_{i}^{\ell,n}}
\newcommand{\Lin}{L_{i,n}}
\newcommand{\Xin}{X_{i,n}}
\newcommand{\llil}{\lambda_{i}^{\ell}}
\newcommand{\lliln}{\lambda_{i}^{\ell,n}}
\newcommand{\biln}{b_{i}^{\ell,n}}
\newcommand{\Li}{L_{i}}
\newcommand{\lli}{\lambda_{i}}
\newcommand{\yyi}{\mathbf{y}_{i}}
\newcommand{\xxil}{\mathbf{x}_{i}^{\ell}}
\newcommand{\bil}{b_{i}^{\ell}}

\let\emptyset\varnothing

\numberwithin{equation}{section}

\title[Lyapunov functions for switched systems]{Convergence of an algorithm for constructing Lyapunov functions for switched systems using meshfree collocation}

\newcommand\blfootnote[1]{%
  \begingroup
  \renewcommand\thefootnote{}\footnote{#1}%
  \addtocounter{footnote}{-1}%
  \endgroup
}

\begin{document}

\author{Jay Ward$^{*1}$, Nicos Georgiou$^1$ and Peter Giesl$^1$}
\address{$^1$University of Sussex, United Kingdom}
\email{Jay.Ward@sussex.ac.uk, N.Georgiou@sussex.ac.uk and P.A.Giesl@sussex.ac.uk}



\blfootnote{*Corresponding author: Jay Ward.}

\begin{abstract}
    Switched systems are a class of dynamical systems where trajectories switch between different systems based on a switching rule. This rule can depend on time and/or the position of the trajectory in the state space. The existence of a Lyapunov function implies the existence of a uniformly asymptotically stable equilibrium point at the origin, which we prove in this paper. A method to construct Lyapunov functions for switched systems using meshfree collocation and quadratic programming was described in a related article. We prove that, under suitable assumptions, the algorithm described in this previous work converges as the fill distance between the collocation points tends to zero.
\end{abstract}

\maketitle



\section{Introduction}
\label{secIntro}

Lyapunov functions are a useful tool to show stability properties of a dynamical system and to determine the basin of attraction of equilibria. When a dynamical system is expressed as a system of differential equations, a (strict) Lyapunov function is a positive function that decreases along solutions of the system. For a smooth function this is equivalent to having a negative orbital derivative. The existence of a Lyapunov function implies that the system has an equilibrium point that is stable and attractive. The basin of attraction of the equilibrium point can be determined through sublevel sets of the Lyapunov function. For background, see \cite{zubov1961methods}, \cite{hahn1967stability} and \cite{KhalilNonlinear}.\par
A construction method for Lyapunov functions for autonomous ODEs is developed in \cite{giesl2007meshless}. Meshfree collocation is used to find an approximation to the PDE $DV(\mathbf{x})=-\|\mathbf{x}\|^{2}$, where $D$ represents the orbital derivative. To do this, finitely many collocation points are selected and the norm-minimal function such that the PDE is satisfied at these points is found. This method is extended in \cite{giesl2018construction} where an approximate solution to a system of partial differential inequalities is found using quadratic programming. The approximate solution has been shown to converge to the solution of the system of partial differential inequalities as the fill distance decreases \cite{giesl2021minimization}.\par
It is useful to construct Lyapunov functions for a variety of dynamical systems, including switched systems. Switched systems are a family of dynamical systems where a switching rule determines which system is \enquote{switched on}. This switching rule can depend on time and/or position in the state space. Switched systems where time-dependent switching takes place are called arbitrary switched systems, and systems where state-dependent switching occurs are called variable structure systems. These systems have widespread applications in science, engineering and economics, e.g. water quality control, manufacturing systems, or aircraft control. An overview of switched systems can be found in \cite{liberzon2003switching} and a review of results related to the stability of such systems can be found in \cite{davrazos2001review}.\par
Many methods have been developed to construct Lyapunov functions for switched systems. One such method is the algorithm developed in \cite{hafstein2009algorithm} to construct Continuous and Piecewise Affine (CPA) Lyapunov functions for arbitrary switched systems. This was then developed further to construct CPA Lyapunov functions for strongly asymptotically stable differential inclusions in \cite{baier2010computing} and \cite{baier2012linear}, and control CPA Lyapunov functions for weakly asymptotically stable differential inclusions in \cite{baier2014numerical}. The algorithm was advanced in \cite{hafstein2020cpa} and \cite{hafstein2022sliding} and applied to variable structure examples. A different algorithm has been developed to construct Lyapunov functions for switched systems in \cite{ward_construction_nodate} using meshfree collocation and quadratic programming. 
A usual assumption in the literature is that the switched system follows one switching rule; however, the switched system defined in \cite{ward_construction_nodate} allows switching based on time and/or position in the state space. To construct a Lyapunov function for this system, an approximate solution to a system of partial differential inequalities $D_{i}V(\mathbf{x})\leq{b_{i}(\mathbf{x})}$ is found, where $i$ indicates the system that is \enquote{switched on}, $D_{i}$ is the orbital derivative with respect to $i$ and $b_{i}$ is a non-positive function which is zero at the equilibrium. In this paper, we show that under suitable assumptions the approximate solution converges to a solution $V$ of the system of partial differential inequalities.\par
\subsection*{Overview of the paper:} This paper provides the theoretical backbone of the algorithm to construct a Lyapunov function for a switched system that appears in \cite{ward_construction_nodate}. We show that this algorithm converges as the fill distance decreases if there exists a Lyapunov function for the system. Section \ref{secSwitchSyst} defines this switched system which allows for time- and state-dependent switching and a corresponding Lyapunov function, and introduces sufficient conditions for the equilibrium point of the system to be uniformly asymptotically stable. In Section \ref{secMinProbi} we examine a system of partial differential inequalities, the solution of which is a Lyapunov function for the switched system, and discretise it to find a minimisation problem (\ref{eqnMinProbi}). The solution to this minimisation problem is shown to converge under certain assumptions when the fill distance (\ref{eqnthmConvergence3}) decreases (i.e. the discretisation gets finer) in Section \ref{secConvergence}. In Section \ref{secConstruction}, we show equivalence of (\ref{eqnMinProbi}) to the minimisation problem from \cite{ward_construction_nodate} which means the convergence theorem from Section \ref{secConvergence} applies to the algorithm in \cite{ward_construction_nodate}. This section also contains details of the method described in \cite{ward_construction_nodate} to construct a Lyapunov function for a switched system; this is summarised in Algorithm \ref{algMain}. Finally, Section \ref{secExamples} gives examples where Algorithm \ref{algMain} is applied to two systems. In the following table we give our main results with brief descriptions and where they can be located.\medskip\\
\begin{tabular}{p{\linewidth}}
     \textbf{Section \ref{secStability}}, Theorem \ref{thmStability}:\\
     Existence of a Lyapunov function for the switched system implies uniform asymptotic stability of its equilibrium point.\medskip\\
     \textbf{Section \ref{secConvergence}}, Theorem \ref{thmConvergence}:\\
     The solution to the minimisation problem (\ref{eqnMinProbi}) converges to a Lyapunov function for the switched system as the fill distance decreases if a Lyapunov function for the system exists.\medskip\\
     \textbf{Section \ref{secFillDistance}}, Corollary \ref{corFillDistance}:\\
     A way to choose the collocation points is given that is a sufficient condition for the fill distance assumption from Theorem \ref{thmConvergence} to be satisfied.\medskip\\
     \textbf{Section \ref{secMinProb}}, Lemma \ref{lemEquivMinProbs}:\\
     The minimisation problem (\ref{eqnMinProbi}) is equivalent to the minimisation problem (8) from \cite{ward_construction_nodate} and thus the algorithm from \cite{ward_construction_nodate} converges.
\end{tabular}

\subsection*{Acknowledgements} 
Jay Ward acknowledges support by the EPSRC, via a PhD studentship funding 2889464. Nicos Georgiou and Peter Giesl gratefully acknowledge partial support from the Dr Perry James (Jim) Browne Research Centre at the Department of Mathematics, University of Sussex.

\section{Switched systems}
\label{secSwitchSyst}

In this paper, we consider switched systems that allow arbitrary switching, variable structure switching, or a combination of the two. We define a corresponding Lyapunov function and state the conditions needed for uniform asymptotic stability of the equilibrium point. The switched system we are considering is defined in (\ref{eqnfis}) in the same way as in \cite{ward_construction_nodate}. Systems are defined on subsets of the state space, with the union of these subsets making up the whole state space. The subsets may overlap; wherever this occurs, the system may switch arbitrarily to a different trajectory. This set-up was discussed in \cite{liberzon2003switching}, and it was stated that this switched system will have an asymptotically stable equilibrium point if there exists a Lyapunov function for each system in the region where it is defined. This allows multiple Lyapunov functions to be considered, as done in \cite{branicky2002multiple} and  \cite{hou1996stability}. On the other hand, \cite{ward_construction_nodate} states a theorem that the existence of a single Lyapunov function for a switched system implies that it has an asymptotically stable equilibrium point. We prove a stronger statement within this section: the existence of a Lyapunov function implies uniform asymptotic stability of the equilibrium point.\par
In Section \ref{secProblemForm} we define a switched system and a Lyapunov function for the system. We also list the assumptions we have on this system in Assumptions \ref{assumpsMain}. In the following Section \ref{secStability}, we define what is means for this switched system to have a uniformly asymptotically stable equilibrium point, and show that a sufficient condition for uniform asymptotic stability is the existence of a Lyapunov function. This result is given in Theorem \ref{thmStability}.

\subsection{Problem formulation}
\label{secProblemForm}

Let $\Omega\subseteq\bR^{d}$ be a bounded domain with Lipschitz boundary and define a switched system by autonomous systems of differential equations
\begin{equation}
    \label{eqnfis}
    \dot{\mathbf{x}}(t)=\mathbf{f}_{i}(\mathbf{x}(t))\text{, }i\in{I}:=\{1,\dots,k\}
\end{equation}
where $\mathbf{x}(t)\in\Omega\subseteq\bR^{d}$ for $t$ in the maximal interval of existence. For each $i\in{I}$ we assume that $\mathbf{f}_i$ is the restriction to $C_{i}$ of a $C^{\alpha-1}$ vector field defined on an open neighbourhood of $C_{i}$ where $\alpha>d/2+2$, and $|I|$ is finite. Assume that $\mathbf{0}\in\Omega$, ${C_{i}}\subseteq\Omega$ and that the sets $C_{i}$ are closed for all $i\in{I}$. We also assume that for the values $i$ where $\mathbf{0}\in{C_{i}}$ the functions $\mathbf{f}_{i}$ have an equilibrium point at the origin, i.e. $\mathbf{f}_{i}(\mathbf{0})=\mathbf{0}$. 
The sets $C_{i}$ are also a covering of $\Omega$, i.e. 
\begin{equation}
    \label{eqnCoveringCi}
    \bigcup_{i\in{I}}C_{i}=\Omega.
\end{equation}\par
\begin{rem}
    \label{remSolsInOmega}
    We assume that $\mathbf{x}(t)\in\Omega$ for $t$ in the maximal interval of existence since we are mainly interested in the basin of attraction (see Definition \ref{defnStability}) of the equilibrium point at the origin. Solutions starting in the basin of attraction will be defined for all positive times.
\end{rem}
The index $i$ in (\ref{eqnfis}) can change depending on time and/or position in the state space. The sets $C_{i}$ enable us to examine systems where time- and state-dependent switching can occur. Within the intersections of these sets the system switches arbitrarily between the functions that are defined. If the boundary of one of these intersections is reached, a state-dependent switch must take place if the function that is currently \enquote{switched on} is undefined in the new subset.\par
To illustrate this, consider the switched system
\begin{equation}
    \label{eqnExampleSwitchedTrajectory}
    \mathbf{\dot{x}}=\mathbf{f}_{i}(\mathbf{x})\text{, }i\in\{1,2\}
\end{equation}
and a potential partition of the domain $\Omega$ as given in Figure \ref{fig:PartitionExampleCis}. In the set $C_{1}\cap{C_{2}}$ both functions $\mathbf{f}_{1}$ and $\mathbf{f}_{2}$ are defined and the system may switch arbitrarily between them. Suppose the trajectory reaches the left boundary and moves into the subset $C_{1}\cap{C_{2}^{c}}$. If the index $i$ is $2$ it will change to $1$; if the index is already $1$ it will stay the same.
\begin{figure}[h]
    \centering

\tikzset{every picture/.style={line width=0.75pt}} 

\begin{tikzpicture}[x=0.75pt,y=0.75pt,yscale=-1,xscale=1]

\draw  [fill={rgb, 255:red, 74; green, 144; blue, 226 }  ,fill opacity=0.33 ] (18.36,40.85) .. controls (18.9,41.11) and (237.17,40.27) .. (238.64,41.74) .. controls (240.11,43.21) and (116.62,68.2) .. (122.5,100.55) .. controls (128.38,132.89) and (250.4,141.71) .. (243.05,169.64) .. controls (235.7,197.58) and (204.65,219) .. (203.51,219) .. controls (202.38,219) and (17.02,217.86) .. (18.36,218.87) .. controls (19.7,219.88) and (17.82,40.59) .. (18.36,40.85) -- cycle ;
\draw  [fill={rgb, 255:red, 208; green, 2; blue, 27 }  ,fill opacity=0.32 ] (18.36,41.85) .. controls (19.59,42.74) and (152.42,42.4) .. (153.49,42.4) .. controls (154.57,42.4) and (286.2,41.86) .. (285.39,41.85) .. controls (284.58,41.84) and (286.2,220.42) .. (285.39,219.87) .. controls (284.58,219.33) and (175.43,219.16) .. (173.96,219.16) .. controls (172.49,219.16) and (85.75,175.05) .. (96.04,133.89) .. controls (106.33,92.72) and (18.12,83.9) .. (18.12,83.9) .. controls (18.12,83.9) and (17.13,40.96) .. (18.36,41.85) -- cycle ;
\draw  [fill={rgb, 255:red, 74; green, 144; blue, 226 }  ,fill opacity=0.32 ][line width=0.75]  (303.33,54.77) -- (332.73,54.77) -- (332.73,84.17) -- (303.33,84.17) -- cycle ;
\draw  [fill={rgb, 255:red, 208; green, 2; blue, 27 }  ,fill opacity=0.3 ][line width=0.75]  (303.33,113.57) -- (332.73,113.57) -- (332.73,142.97) -- (303.33,142.97) -- cycle ;
\draw  [line width=1.5]  (18.36,40.76) -- (285.39,40.76) -- (285.39,218.78) -- (18.36,218.78) -- cycle ;
\draw  [fill={rgb, 255:red, 74; green, 144; blue, 226 }  ,fill opacity=0.32 ][line width=0.75]  (303.33,113.57) -- (332.73,113.57) -- (332.73,142.97) -- (303.33,142.97) -- cycle ;
\draw  [fill={rgb, 255:red, 208; green, 2; blue, 27 }  ,fill opacity=0.3 ][line width=0.75]  (303.33,172.38) -- (332.73,172.38) -- (332.73,201.78) -- (303.33,201.78) -- cycle ;

\draw (334.73,116.97) node [anchor=north west][inner sep=0.75pt]  [font=\normalsize]  {$=C_{1} \cap C_{2}$};
\draw (265.79,25) node [anchor=north west][inner sep=0.75pt]    {$\Omega $};
\draw (334.73,58.17) node [anchor=north west][inner sep=0.75pt]  [font=\normalsize]  {$=C_{1} \cap C_{2}^{c}$};
\draw (334.73,175.78) node [anchor=north west][inner sep=0.75pt]  [font=\normalsize]  {$=C_{1}^{c} \cap C_{2}$};

\end{tikzpicture}

    \caption{The state space $\Omega$ divided into three subsets $C_{1}\cap{C_{2}^{c}}$, $C_{1}\cap{C_{2}}$ and ${C_{1}^{c}}\cap{C_{2}}$, where $C_{1}$ and $C_{2}$ are the sets where the corresponding functions $\mathbf{f}_{1}$ and $\mathbf{f}_{2}$ from system (\ref{eqnExampleSwitchedTrajectory}) are defined, respectively.}
    \label{fig:PartitionExampleCis}
\end{figure}
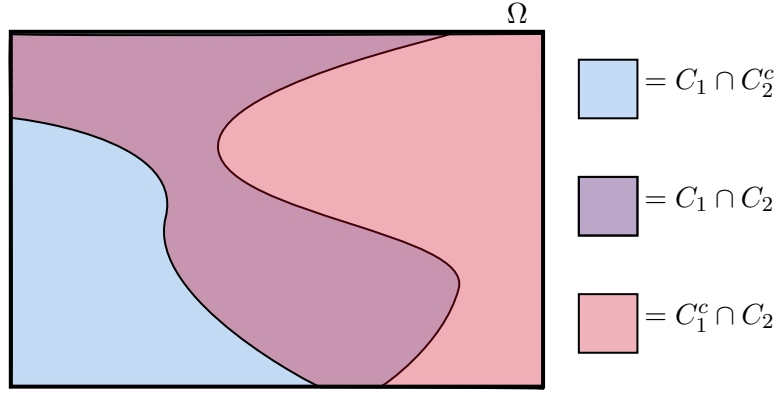\par
We indicate how $i$ changes in the system (\ref{eqnfis}) using the following (anchored) switching sequence indexed by an initial state $\mathbf{x}\in\Omega$
\begin{equation}
    \label{eqnSwitchingSeqWholeSet}
    \sigma:=\{(i_{0},0),(i_{1},t_{1}),\dots,(i_{j},t_{j}),\dots:i_{j}\in{I},{j}\in\bN_{0}\}.
\end{equation}
This switching sequence completely describes a trajectory of the system (\ref{eqnfis}) according to the following rule: $(i_{j},t_{j})$ means that the system evolves according to $\dot{\mathbf{x}}(t)=\mathbf{f}_{i_{j}}(\mathbf{x}(t))$ for $t_{j}\leq{t}<t_{j+1}$. This can be used to describe a switched system that allows for both time- and state-dependent switching, as some of these times $t_{j}$ can correspond to the time at which a state change that causes the system to switch takes place. For $t>0$ we define
\begin{equation}
    \label{eqnShiftOperator}
    \theta_{t}(\sigma):=\{(i_{n-1},t),(i_{n},t_{n}-t),\dots,(i_{j},t_{j}-t),\dots:i_{j}\in{I},j\geq{n-1},t_{n-1}\leq{t}<{t_{n}}\}
\end{equation}
to be a shift operator, where $\theta_{t}(\sigma)$ is the switching sequence $\sigma$ after time $t$ has passed. Set $\Sigma_{\mathbf{x}}$ to be the set of all switching sequences that can occur when starting from the initial state $\mathbf{x}\in\Omega$. A possible system trajectory is denoted by $\mathbf{x}_{\sigma}(\cdot)$ for $\sigma\in\Sigma_{\mathbf{x}}$.
We can define $I_{j}:=[t_{j},t_{j+1})$ for $j\in\bN_{0}$ to be the interval upon which only one system $\dot{\mathbf{x}}(t)=\mathbf{f}_{i_{j}}(\mathbf{x}(t))$ is \enquote{switched on}. We assume that in any finite interval $[a,b]$ with $0<a<b$ that $|\theta_{a}(\sigma)\setminus\theta_{b}(\sigma)|$ is finite, i.e. there is a finite selection of indices $i\in{I}$ that are \enquote{switched on} in the interval $[a,b]$, see (7) from Assumptions \ref{assumpsMain}.\par
We now list the assumptions that we have made in the set-up of our switched system. 
\begin{assumps}
    \label{assumpsMain}
    \hfill
    \begin{enumerate}
        \item For each $i\in I$ we assume that $\mathbf{f}_i$ is the restriction to $C_{i}$ of a $C^{\alpha-1}$ vector field defined on an open neighbourhood of $C_{i}$ where $\alpha>d/2+2$.
        \item $C_{i}\subseteq\Omega$ and the sets $C_{i}$ are closed for all $i\in{I}$.
        \item The sets $C_{i}$ for all $i\in{I}$ are a covering of $\Omega$.
        \item If $\mathbf{0}\in{C_{i}}$ then $\mathbf{f}_{i}(\mathbf{0})=\mathbf{0}$.
        \item $\Omega\subseteq{\bR^{d}}$ is a bounded domain with Lipschitz boundary with $\mathbf{0}\in\Omega^{\circ}$.
        \item $|I|$ is finite.
        \item Switching sequences are chosen so that only finitely many switches occur in finite time.
    \end{enumerate}
\end{assumps}\par

A Lyapunov function of the switched system (\ref{eqnfis}) is a function that is non-increasing along the solutions of each differential equation in the system. For our purposes we say that $V\in{C^{1}(\Omega)}$, $V(\mathbf{0})=0$ is a Lyapunov function for the system (\ref{eqnfis}) if and only if
\begin{equation}
    \label{eqnLyapSwitch}
    \nabla{V(\mathbf{x})}\cdot\mathbf{f}_{i}(\mathbf{x})\leq-\gamma(\|\mathbf{x}\|)\text{ and }V(\mathbf{x})>0
\end{equation}
holds for all $\mathbf{x}\in\Omega$ and all $i\in{I}$ such that $\mathbf{x}\in{C_{i}}\setminus \{\mathbf{0}\}$, where $\gamma(r)>0$ for all $r\neq{0}$. The function $V$ decreases along solutions of each differential equation in the system and is zero at zero, meaning that this is actually a strict Lyapunov function for the system (\ref{eqnfis}).\par
The function $V$ from (\ref{eqnLyapSwitch}) will not just be decreasing along the solutions of each differential equation in the system (\ref{eqnfis}), it will also be decreasing along solution trajectories $\mathbf{x}(t)$ of the system as a whole. To show this, we use the switching sequence defined in (\ref{eqnSwitchingSeqWholeSet}) to describe the solution trajectories of the system (\ref{eqnfis}).\par
\begin{lemma}
    \label{lemDecreaseV}
    The function $V$ satisfying (\ref{eqnLyapSwitch}) is decreasing along solution trajectories $\mathbf{x}(t)$ of the system (\ref{eqnfis}) if $\mathbf{x}(0)\neq\mathbf{0}$.
\end{lemma}
\begin{proof}
    Choose $0<T_{1}<T_{2}$, $j\in\bN_{0}$ and assume that $\mathbf{x}(T_{1}),\mathbf{x}(T_{2})\neq\mathbf{0}$. Recall that $I_{j}:=[t_{j},t_{j+1})$ for $j\in\bN_{0}$ are intervals upon which only one system $\dot{\mathbf{x}}(t)=\mathbf{f}_{i_{j}}(\mathbf{x}(t))$ from (\ref{eqnfis}) is \enquote{switched on}, where $t_{j}$ and $t_{j+1}$ are switching times (see (\ref{eqnSwitchingSeqWholeSet})). We wish to show that in both cases where $T_{1},T_{2}\in{I_{j}}$ and $T_{1}\in{I_{j}}$, $T_{2}\notin{I_{j}}$ we have $V(\mathbf{x}(T_{1}))>V(\mathbf{x}(T_{2}))$.\par
First, we show this for $T_{1},T_{2}\in{I_{j}}$. If we integrate (\ref{eqnLyapSwitch}) on the interval $[T_{1},T_{2}]$ we have
\begin{equation}
    \label{eqnIntegrateOrbDerv}
    \int_{T_{1}}^{T_{2}}\nabla{V(\mathbf{x}(t))}\cdot\mathbf{f}_{i_{j}}(\mathbf{x}(t))dt=\int_{\mathbf{x}(T_{1})}^{\mathbf{x}(T_{2})}\nabla{V(\mathbf{x}})\cdot{d\mathbf{x}}=V(\mathbf{x}(T_{2}))-V(\mathbf{x}(T_{1}))
\end{equation}
where the first equality uses that $d\mathbf{x}/dt=\mathbf{f}_{i_{j}}(\mathbf{x})$. The integral of the right-hand-side of (\ref{eqnLyapSwitch}) is always negative (except at $\mathbf{0}$) due to $\gamma$ being positive. Consequently, the difference found in (\ref{eqnIntegrateOrbDerv}) is negative and $V(\mathbf{x}(T_{1}))>V(\mathbf{x}(T_{2}))$.\par
Now we use this to show that if $T_{1}\in{I_{j}}$, $T_{2}\notin{I_{j}}$ then the inequality $V(\mathbf{x}(T_{1}))>V(\mathbf{x}(T_{2}))$ still holds. Integrating (\ref{eqnLyapSwitch}) on the interval $I_{j}$ is done in the same way as (\ref{eqnIntegrateOrbDerv}), and we find that $V(\mathbf{x}(t_{j}))>V(\mathbf{x}(t_{j+1}))$. As we have that $T_{2}$ is not in the interval $I_{j}$ we can assume that it is a member of a different interval $I_{\ell}$ where $j+1\leq\ell$, $\ell\in\bN_{0}$. This means that
\begin{equation}
    \label{eqnDecreasingV}
    {V(\mathbf{x}(T_{1}))}>{V(\mathbf{x}(t_{j+1}))}\geq{V(\mathbf{x}(t_{\ell}))}\geq{V(\mathbf{x}(T_{2}))}
\end{equation}
where we have $V(\mathbf{x}(t_{\ell}))\geq{V(\mathbf{x}(T_{2}))}$ using (\ref{eqnIntegrateOrbDerv}) and that $T_{2}\in{I_{\ell}}$. This gives us that $V(\mathbf{x}(T_{1}))>V(\mathbf{x}(T_{2}))$.
Therefore, due to the properties of $V$ as described in (\ref{eqnLyapSwitch}), it strictly decreases along solution trajectories $\mathbf{x}(t)$ of the system (\ref{eqnfis}).
\end{proof}

In the following section we find the sufficient conditions for the uniform asymptotic stability of the equilibrium point $\mathbf{0}$ of the system (\ref{eqnfis}). 

\subsection{Asymptotic stability}
\label{secStability}

In this section, we prove that the existence of a Lyapunov function of the form (\ref{eqnLyapSwitch}) for all $i\in{I}$ ensures that the equilibrium point at the origin of system (\ref{eqnfis}) is uniformly asymptotically stable. This is a new variation on previous results for switched systems as two switching rules can be incorporated and the asymptotic stability is uniform. We prove this by adapting classical Lyapunov methods in Theorem \ref{thmStability}; however, before this we introduce Lemma \ref{lemStableInvariantNeighbourhood} which is necessary to complete this proof.\par
The $\sigma$-flow is a continuous mapping $(\mathbf{x},t)\mapsto{S}_{\sigma}(\mathbf{x},t)$ with $\mathbf{x}\in\Omega$, $t>0$ and $\sigma\in\Sigma_{\mathbf{x}}$ which maps the state of the system (\ref{eqnfis}) at time $0$ to the state at time $t$ dependent on the switching sequence $\sigma$, i.e. $S_{\sigma}(\mathbf{x},t)=\mathbf{x}_{\sigma}(t)$. The flow for sets $K\subseteq{\Omega}$ is the set of all points which are reached after a time $t$ from any starting point $\mathbf{x}\in{K}$ under the $\sigma$-flow for all $\sigma\in\Sigma_{\mathbf{x}}$. This can be written as
\begin{equation}
    \label{eqnFlowOperatorK}
    S(K,t):=\{S_{\sigma}(\mathbf{x},t):\mathbf{x}\in{K},\sigma\in\Sigma_{\mathbf{x}}\}.
\end{equation}
The $\sigma$-flow has the properties that $S_{\sigma}(\mathbf{x},0)=\mathbf{x}$ and $S_{\sigma}(\mathbf{x},t+s)=S_{\theta_{s}(\sigma)}(S_{\sigma}(\mathbf{x},s),t)$, where $\theta_{s}$ is the shift operator defined in (\ref{eqnShiftOperator}). Using this property we can also show that $S(K,t+s)=S(S(K,s),t)$. Now, using this notation, we introduce the definition of a positively invariant set.
\begin{defn}{\rm (Positively invariant set)}
    \label{defnInvariance}
    Let $K\subseteq{\Omega}$. We say $K$ is positively invariant if $S(K,t)\subseteq{K}$ for all $t\geq{0}$.
\end{defn}
We can now define the uniform asymptotic stability of the equilibrium point $\mathbf{x}_{0}=\mathbf{0}$ for the system (\ref{eqnfis}).
\begin{defn}{\rm (Uniform asymptotic stability and basin of attraction)}
    \label{defnStability}
    The equilibrium point $\mathbf{x}_{0}=\mathbf{0}$ of system (\ref{eqnfis}) is stable if for all $\varepsilon>0$ there exists a $\delta>0$ such that for all $t\geq{0}$ the solution exists and
    \begin{equation}
        \label{eqndefnStability1}
        \|\mathbf{x}\|<\delta\Rightarrow\forall\sigma\in\Sigma_{\mathbf{x}},\phantom{x}\|S_{\sigma}(\mathbf{x},t)\|<\varepsilon.
    \end{equation}
    The equilibrium point $\mathbf{x}_{0}=\mathbf{0}$ of system (\ref{eqnfis}) is uniformly asymptotically stable if it is stable and there exists a $\delta'>0$ such that for all $\varepsilon>0$ there exists a $T>0$ such that for all times $t>T$, we have that
    \begin{equation}
        \label{eqndefnStability2}
        \|\mathbf{x}\|<\delta'\Rightarrow\forall\sigma\in\Sigma_{\mathbf{x}},\phantom{x}\|S_{\sigma}(\mathbf{x},t)\|<\varepsilon.
    \end{equation}
    The set
    \begin{equation}
        \label{eqndefnStability3}
        A(\mathbf{0}):=\{\mathbf{x}\in\Omega:\forall\sigma\in\Sigma_{\mathbf{x}}\text{, }\lim_{t\rightarrow\infty}S_{\sigma}(\mathbf{x},t)=\mathbf{0}\}
    \end{equation}
    is called the basin of attraction of $\mathbf{x}_{0}=\mathbf{0}$, where $\mathbf{x}_{0}$ is a uniformly asymptotically stable equilibrium point.
\end{defn}

\begin{lemma}
    \label{lemStableInvariantNeighbourhood}
    Let $\mathbf{x}_{0}=\mathbf{0}$ be the equilibrium point of the system (\ref{eqnfis}). The following two statements are equivalent:
    \begin{enumerate}
        \item[(i)] $\mathbf{x}_{0}$ is stable.
        \item[(ii)] For all neighbourhoods $K$ of $\mathbf{x}_{0}$ there is a neighbourhood $K'\subseteq{K}$ of $\mathbf{x}_{0}$ which is positively invariant.
    \end{enumerate}
\end{lemma}
\begin{proof}
    We can write the stability property (\ref{eqndefnStability1}) in a way that is more relevant for this proof. Define 
    \begin{equation}
        \label{eqnlemprfStableInvariantNeighbourhood1}
        I_{\varepsilon}:=\{\mathbf{x}\in\Omega:\|\mathbf{x}\|<\varepsilon\}
    \end{equation}
    to be an $\varepsilon$-neighbourhood of the origin. Then $\mathbf{x}_{0}=\mathbf{0}$ is stable if and only if for every neighbourhood $I_{\varepsilon}$ there exists a neighbourhood $I_{\delta}$ with $I_{\delta}\subseteq{I_{\varepsilon}}$ such that $S(I_{\delta},t)\subseteq{I_{\varepsilon}}$ for all $t\geq{0}$.\medskip\\
    (i)$\Rightarrow$(ii): As we are assuming that $\mathbf{x}_{0}=\mathbf{0}$ is stable, the property described below (\ref{eqnlemprfStableInvariantNeighbourhood1}) holds true. Let $K$ be a neighbourhood of $\mathbf{x}_{0}$ and $\varepsilon>0$ be small enough such that $I_{\varepsilon}\subseteq{K}$. Define $K'=\cup_{t\geq{0}}S(I_{\delta},t)\subseteq{I_{\varepsilon}}\subseteq{K}$. The set $K'$ is a neighbourhood of $\mathbf{x}_{0}$ as $I_{\delta}\subseteq{K'}$. Then, using that $S(K,t+s)=S(S(K,s),t)$, we have
    \begin{equation}
        \label{eqnlemprfStableInvariantNeighbourhood2}
        S(K',s)=S\left(\bigcup_{t\geq{0}}S(I_{\delta},t),s\right)=\bigcup_{t\geq{s}}S(I_{\delta},t)\subseteq{K'}\text{ for all }s\geq{0}.\nonumber
    \end{equation}
    So there is a neighbourhood $K'$ of $\mathbf{x}_{0}$ that is positively invariant.\medskip\\
    (ii)$\Rightarrow$(i): We now assume that property (ii) holds and wish to show that $\mathbf{x}_{0}=\mathbf{0}$ is stable. As for all neighbourhoods $K$ we can find a $K'\subseteq{K}$ that is positively invariant, for any $I_{\varepsilon}:=K$ there must exist a positively invariant subset $K'\subseteq{I_{\varepsilon}}$ that is also a neighbourhood of $\mathbf{x}_{0}$. Since $K'$ is a neighbourhood of $\mathbf{x}_{0}$, there exists $\delta>0$ such that $I_\delta\subseteq K'$. Since $K'$ is positively invariant,
    \begin{equation}
        \label{eqnlemprfStableInvariantNeighbourhood3}
        S(I_\delta,t)\subseteq S(K',t)\subseteq K'\subseteq I_\varepsilon\text{ for all }t\geq{0}.\nonumber
    \end{equation}
    Therefore, we have shown that for every neighbourhood $I_{\varepsilon}$ of $\mathbf{x}_{0}$ there exists a neighbourhood $I_{\delta}$ of $\mathbf{x}_{0}$ such that $S(I_{\delta},t)\subseteq{I_{\varepsilon}}$ for all $t\geq{0}$. This tells us that the equilibrium point $\mathbf{x}_{0}$ is stable.
\end{proof}\par

We can now prove the following main result which shows that the existence of a Lyapunov function implies that the equilibrium point $\mathbf{x}_{0}=\mathbf{0}$ of the system (\ref{eqnfis}) is uniformly asymptotically stable.
\begin{theorem}
    \label{thmStability}
    If there exists a function $V\in{C^{1}(\Omega)}$ that satisfies properties (\ref{eqnLyapSwitch}) for all $i\in{I}$, then the equilibrium $\mathbf{x}_{0}=\mathbf{0}$ of system (\ref{eqnfis}) is uniformly asymptotically stable. Also, let $R>0$ be such that the level set
    \begin{equation}
        \label{eqnthmStability1}
        S_{R}:=\{\mathbf{x}\in\Omega:V(\mathbf{x})\leq{R}\}\subseteq{\Omega^{\circ}}
    \end{equation}
    is compact. Then $S_{R}$ is a subset of the basin of attraction of $\mathbf{x}_{0}$.
\end{theorem}
\begin{proof}
    This proof is separated into two parts: first stability, and then uniform asymptotic stability. In the process of doing this we show that if $S_{R}$ satisfies (\ref{eqnthmStability1}) and is compact then it is a subset of the basin of attraction of $\mathbf{x}_{0}$, we briefly describe how this has been done at the end of the proof.\medskip\par
    We can find an $\varepsilon_{0}$ such that $I_{\varepsilon_{0}}$ (defined in (\ref{eqnlemprfStableInvariantNeighbourhood1})) satisfies $I_{\varepsilon_{0}}\subseteq\Omega$. Fix $0<\varepsilon\leq\varepsilon_{0}$ and let
    \begin{equation}
        \label{eqnthmprfStability1}
        \mu=\min\{V(\mathbf{x}):\|\mathbf{x}\|=\varepsilon\}>0.
    \end{equation}
    Since $V$ is a continuous function with $V(\mathbf{x})=0$ if and only if $\mathbf{x}=\mathbf{0}$, there is a $\delta>0$ (with $\delta\leq\varepsilon$) such that $V(\mathbf{x})<\mu$ for all $\mathbf{x}\in{I_{\delta}}$.\par
    Now for $\mathbf{x}_{0}=\mathbf{0}$ to be stable we want to show that $\mathbf{x}\in{I_{\delta}}$ implies $S_{\sigma}(\mathbf{x},t)\in{I_{\varepsilon}}$ for all $\sigma\in\Sigma_{\mathbf{x}}$ and all $t\geq{0}$. Assume that this is false. Then there is an $\mathbf{x}\in{I_{\delta}}$, a time $t_{\varepsilon}\geq{0}$ and a $\sigma\in\Sigma_{\mathbf{x}}$ such that $S_{\sigma}(\mathbf{x},t)$ is in $I_{\varepsilon}$ up to the time $t_{\varepsilon}$, when it lies on the boundary, i.e. $\|S_{\sigma}(\mathbf{x},t)\|<\varepsilon$ for all $t\in[0,t_{\varepsilon})$, $\|S_{\sigma}(\mathbf{x},t_{\varepsilon})\|=\varepsilon$. Since (\ref{eqnLyapSwitch}) implies that $V$ is decreasing along solution trajectories as was shown in Lemma \ref{lemDecreaseV}, $V(\mathbf{x})<\mu$ implies that $V(S_{\sigma}(\mathbf{x},t_{\varepsilon}))<\mu$. This contradicts (\ref{eqnthmprfStability1}) since $\|S_{\sigma}(\mathbf{x},t_{\varepsilon})\|=\varepsilon$.\medskip\par
    Now that we have shown stability of the equilibrium point $\mathbf{x}_{0}=\mathbf{0}$, we wish to show that it is also uniformly asymptotically stable. We do this by way of contradiction. We assume the opposite of (\ref{eqndefnStability2}) which means that
    \begin{itemize}
        \item[(1)] for all neighbourhoods $I_{\delta'}$ of $\mathbf{x}_{0}$ there exists a neighbourhood $I_{\varepsilon}\subseteq{I_{\delta'}}$ of $\mathbf{x}_{0}$ where for all $T>0$ there exists a time $t>T$, $\mathbf{x}'\in{I_{\delta'}}$ and a $\sigma\in\Sigma_{\mathbf{x}'}$ with $S_{\sigma}(\mathbf{x}',t)\notin{I_{\varepsilon}}$.
    \end{itemize}
    Here we have assumed that $\varepsilon\leq\delta'$ as the proof of the alternative is trivial due to the stability of $\mathbf{x}_{0}$. As we know that there exists a Lyapunov function $V$ that satisfies properties (\ref{eqnLyapSwitch}), we know that this $V$ is positive except at the origin (and decreases along solution trajectories). Therefore, we can find a level set $S_{R}$ (\ref{eqnthmStability1}) that is a compact subset of $\Omega$. Now we choose $\delta'>0$ such that $I_{\delta'}\subseteq{S_{R}}$. By assumption (1) there exists an $\varepsilon$ with $I_{\varepsilon}\subseteq{I_{\delta'}}$. Using Lemma \ref{lemStableInvariantNeighbourhood} we find a positively invariant neighbourhood of $\mathbf{x}_{0}$ which we label $U$, where $U\subseteq{I_{\varepsilon}}$. Now by the properties of $V$ we know there exists a level set $S_{r}:=\{\mathbf{x}:V(\mathbf{x})\leq{r}\}$ with $r<R$ such that $S_{r}\subseteq{U}$. Finally, define the set $D:=\{\mathbf{x}:r\leq{V(\mathbf{x})}\leq{R}\}$. The set-up of these sets is illustrated in Figure \ref{fig:StabilitySets}.\par
    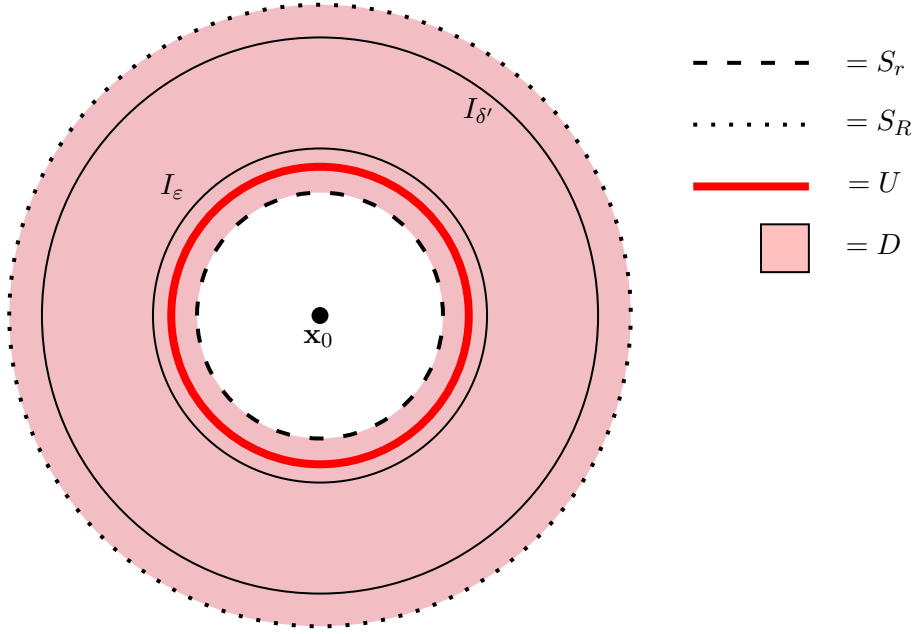
\begin{figure}[h]
        \centering

\tikzset{every picture/.style={line width=0.75pt}} 

\begin{tikzpicture}[x=0.75pt,y=0.75pt,yscale=-1,xscale=1]

\draw  [fill={rgb, 255:red, 208; green, 2; blue, 27 }  ,fill opacity=0.26 ][dash pattern={on 1.5pt off 5.25pt}][line width=1.5]  (154.84,180.75) .. controls (154.84,94.65) and (224.65,24.84) .. (310.75,24.84) .. controls (396.85,24.84) and (466.66,94.65) .. (466.66,180.75) .. controls (466.66,266.85) and (396.85,336.66) .. (310.75,336.66) .. controls (224.65,336.66) and (154.84,266.85) .. (154.84,180.75) -- cycle ;
\draw  [fill={rgb, 255:red, 255; green, 255; blue, 255 }  ,fill opacity=1 ][dash pattern={on 5.25pt off 7.5pt}][line width=1.5]  (249.06,180.75) .. controls (249.06,146.68) and (276.68,119.06) .. (310.75,119.06) .. controls (344.82,119.06) and (372.44,146.68) .. (372.44,180.75) .. controls (372.44,214.82) and (344.82,242.44) .. (310.75,242.44) .. controls (276.68,242.44) and (249.06,214.82) .. (249.06,180.75) -- cycle ;
\draw   (171.25,180.75) .. controls (171.25,103.71) and (233.71,41.25) .. (310.75,41.25) .. controls (387.79,41.25) and (450.25,103.71) .. (450.25,180.75) .. controls (450.25,257.79) and (387.79,320.25) .. (310.75,320.25) .. controls (233.71,320.25) and (171.25,257.79) .. (171.25,180.75) -- cycle ;
\draw  [color={rgb, 255:red, 255; green, 0; blue, 0 }  ,draw opacity=1 ][line width=3]  (236.13,180.75) .. controls (236.13,139.54) and (269.54,106.13) .. (310.75,106.13) .. controls (351.96,106.13) and (385.38,139.54) .. (385.38,180.75) .. controls (385.38,221.96) and (351.96,255.38) .. (310.75,255.38) .. controls (269.54,255.38) and (236.13,221.96) .. (236.13,180.75) -- cycle ;
\draw  [fill={rgb, 255:red, 0; green, 0; blue, 0 }  ,fill opacity=1 ] (307,180.75) .. controls (307,178.68) and (308.68,177) .. (310.75,177) .. controls (312.82,177) and (314.5,178.68) .. (314.5,180.75) .. controls (314.5,182.82) and (312.82,184.5) .. (310.75,184.5) .. controls (308.68,184.5) and (307,182.82) .. (307,180.75) -- cycle ;
\draw [line width=1.5]  [dash pattern={on 5.25pt off 7.5pt}]  (498,54) -- (556,54) ;
\draw [line width=1.5]  [dash pattern={on 1.5pt off 5.25pt}]  (498,85) -- (556,85) ;
\draw  [fill={rgb, 255:red, 255; green, 0; blue, 0 }  ,fill opacity=0.25 ] (532,135) -- (556,135) -- (556,159) -- (532,159) -- cycle ;
\draw   (226.95,180.75) .. controls (226.95,134.47) and (264.47,96.95) .. (310.75,96.95) .. controls (357.03,96.95) and (394.55,134.47) .. (394.55,180.75) .. controls (394.55,227.03) and (357.03,264.55) .. (310.75,264.55) .. controls (264.47,264.55) and (226.95,227.03) .. (226.95,180.75) -- cycle ;
\draw [color={rgb, 255:red, 255; green, 0; blue, 0 }  ,draw opacity=1 ][line width=3]    (498,116) -- (556,116) ;

\draw (381,70.4) node [anchor=north west][inner sep=0.75pt]    {$I_{\delta '}$};
\draw (229,108.4) node [anchor=north west][inner sep=0.75pt]    {$I_{\varepsilon }$};
\draw (301,186.4) node [anchor=north west][inner sep=0.75pt]    {$\mathbf{x}_{0}$};
\draw (573,46.4) node [anchor=north west][inner sep=0.75pt]    {$=S_{r}$};
\draw (573,76.4) node [anchor=north west][inner sep=0.75pt]    {$=S_{R}$};
\draw (573,138.4) node [anchor=north west][inner sep=0.75pt]    {$=D$};
\draw (574,107.4) node [anchor=north west][inner sep=0.75pt]    {$=U$};

\end{tikzpicture}

        \caption{The set-up of the sets described in the proof of Theorem \ref{thmStability}, specifically the section on proving uniform asymptotic stability of the equilibrium point $\mathbf{x}_{0}$. The sets $I_{\delta'}$, $I_{\varepsilon}$ and $U$ are neighbourhoods of $\mathbf{x}_{0}$ where $U$ is positively invariant; $S_{r}$ and $S_{R}$ are level sets of the Lyapunov function $V$.}
        \label{fig:StabilitySets}
    \end{figure}
    Since $S_{R}$ is compact and $S_{r}$ is a neighbourhood of the origin, the set $D$ is compact and does not contain the origin. For each $i$ such that $D\cap C_{i}\ne\varnothing$, the function $x\mapsto -\nabla V(\mathbf{x})\cdot f_{i}(\mathbf{x})$ is continuous and strictly positive on the compact set $D\cap C_{i}$. Hence
    \begin{equation}
        \label{eqnthmprfStability1.1}
        a_{i}:=\min_{\mathbf{x}\in D\cap C_{i}}\bigl(-\nabla V(\mathbf{x})\cdot f_{i}(\mathbf{x})\bigr)>0.\nonumber
    \end{equation}
    Since $I$ is finite, we may define
    \begin{equation}
        \label{eqnthmprfStability1.2}
        A:=\min_{\{i:D\cap C_{i}\ne\varnothing\}} a_{i}>0.\nonumber
    \end{equation}
    Then, for all $i\in{I}$ and $\mathbf{x}\in D\cap C_{i}$,
    \begin{equation}
        \label{eqnthmprfStability2}
        \nabla V(\mathbf{x})\cdot f_{i}(\mathbf{x})\le -A.
    \end{equation}
    This means that for all $\mathbf{x}\in{D}$ we have that $-\gamma(\|\mathbf{x}\|)\leq-A$. Choose the time $T:=2(R-r)/A>0$. Now we have fixed our $\delta'$, $\varepsilon$ and $T$ that we use in (1) and it now becomes
    \begin{itemize}
        \item[($\ast$)] there exists a $t>T$, $\mathbf{x}'\in{I_{\delta'}}$ and a $\sigma\in\Sigma_{\mathbf{x}'}$ such that $S_{\sigma}(\mathbf{x}',t)\notin{I_{\varepsilon}}$.
    \end{itemize}\par
    To obtain a contradiction we consider two separate cases which cover all of the elements of $I_{\delta'}$. The first case we look at is $\mathbf{x}\in{S_{r}}$, and the second case is $\mathbf{x}\in{D}$.\\
    \textit{Case 1:} $\mathbf{x}\in{S_{r}}\subseteq{U}\subseteq{I_{\delta'}}$. As $U$ is a positively invariant set we have that $S_{\sigma}(\mathbf{x},t)\in{U}\subseteq{I_{\varepsilon}}$ for all $\sigma\in\Sigma_{\mathbf{x}}$ and $t\geq{0}$. This gives us a contradiction to ($\ast$) for any $\mathbf{x}\in{S_{r}}$.\\
    \textit{Case 2:} $\mathbf{x}\in{D}$. To find a contradiction to ($\ast$) we prove the following lemma and use the arguments made in Case 1.
    \begin{lemma}
    \label{lemStabProof}
        Let $V\in{C^{1}(\Omega)}$ be a function that satisfies properties (\ref{eqnLyapSwitch}) for all $i\in{I}$. Define $S_{R}$ and $S_{r}$ to be level sets as in (\ref{eqnthmStability1}) with $r<R$, and $D:=\{\mathbf{x}\in\Omega:r\leq{V(\mathbf{x})}\leq{R}\}$. Set $T:=2(R-r)/A$ with $A$ defined as in (\ref{eqnthmprfStability2}). Then for all $\mathbf{x}\in{D}$ and switching sequences $\sigma\in\Sigma_{\mathbf{x}}$ there exists a time $t\leq{T}$ such that $S_{\sigma}(\mathbf{x},t)\in{S_{r}}$.
    \end{lemma}
    \begin{proof}
        We prove this by way of contradiction. Assume that there exists an $\mathbf{x}\in{D}$ and a switching sequence $\sigma\in\Sigma_{\mathbf{x}}$ such that $S_{\sigma}(\mathbf{x},t)\notin{S_{r}}$ for all $t\leq{T}$. Then $V(S_{\sigma}(\mathbf{x},T))>r$. As $\mathbf{x}\in{D}$ then $\mathbf{x}\in{S_{R}}$, which means that $S_{\sigma}(\mathbf{x},t)\in{S_{R}}$ for all $\sigma\in\Sigma_{\mathbf{x}}$ and $t\geq{0}$, as $V$ is decreasing. Then, using our assumption we have $S_{\sigma}(\mathbf{x},t)\in{D}$ for all $t\leq{T}$.
        Using the definition of $T$,
    \begin{equation}
        \label{eqnthmprfStability3}
        V(S_{\sigma}(\mathbf{x},T))>r=R-\frac{1}{2}AT\geq{V(\mathbf{x})-\frac{1}{2}AT}.
    \end{equation}
    Integrating (\ref{eqnLyapSwitch}) with techniques similar to (\ref{eqnIntegrateOrbDerv}) between $0$ and $T$ and using (\ref{eqnthmprfStability2}) we find that for all $\mathbf{x}\in{D}$, $\sigma\in\Sigma_{\mathbf{x}}$ and $S_{\sigma}(\mathbf{x},t)\in{D}$ for $t\in[0,T]$
    \begin{equation}
        \label{eqnthmprfStability4}
        V(S_{\sigma}(\mathbf{x},T))\leq{V(\mathbf{x})-A{T}}.\nonumber
    \end{equation}
    This gives us a contradiction to  (\ref{eqnthmprfStability3}). Therefore, we have shown this lemma to be true. 
    \end{proof}\par
    As $I_{\delta'}\subseteq{S_{r}\cup{D}}=S_{R}$ the combination of Case 1 and 2 contradicts assumption ($\ast$). This means we have a contradiction to assumption (1) and $\mathbf{x}_{0}$ is uniformly asymptotically stable.\medskip\par
    In the process of showing $\mathbf{x}_{0}=\mathbf{0}$ is uniformly asymptotically stable, we have shown for all $\varepsilon>0$ there exists a $T$ such that for all $t>T$, $\mathbf{x}\in{S_{R}}$ and $\sigma\in\Sigma_{\mathbf{x}}$ we have $S_{\sigma}(\mathbf{x},t)\in{I_{\varepsilon}}$. This is a stronger statement than is needed and therefore is enough to prove that the compact level set $S_{R}$ (\ref{eqnthmStability1}) is a subset of the basin of attraction of $\mathbf{x}_{0}=\mathbf{0}$, defined in (\ref{eqndefnStability3}).
\end{proof}

\section{Minimisation problem}
\label{secMinProbi}

We wish to find a minimisation problem where the solution $v$ is an approximation of a function $V$ that satisfies the conditions in (\ref{eqnLyapSwitch}). This allows us to construct a Lyapunov function for the switched system (\ref{eqnfis}). A similar method is used in \cite{giesl2018construction} for autonomous ODEs. There, the authors construct a Lyapunov function for a single system by finding a solution to the system of partial differential inequalities $DV(\mathbf{x})\leq{r}(\mathbf{x})$ where $D$ denotes the orbital derivative, and $r(\mathbf{x})$ is negative. To generalise this method, we define for all $i=1,\dots,k$ the linear differential operator
\begin{equation}
    \label{eqnDi}
    D_{i}V(\mathbf{x}):=\nabla_{\mathbf{x}}
    V(\mathbf{x})\cdot\mathbf{f}_{i}(\mathbf{x})\phantom{xxx}\forall\mathbf{x}\in{C_{i}}
\end{equation}
where for each $i\in{I}$ we assume that $\mathbf{f}_i$ is the restriction to $C_{i}$ of a $C^{\alpha-1}$ vector field defined on an open neighbourhood of $C_{i}$ where $\alpha>d/2+2$. We wish to solve
\begin{equation}
    \label{eqnDiV}
    D_{i}V(\mathbf{x})\leq{b_{i}}(\mathbf{x})
\end{equation}
where $b_{i}(\mathbf{x})$ are continuous functions on $C_{i}$ that are negative for all $\mathbf{x}\in{C_{i}}\setminus{\{\mathbf{0}\}}$, $b_{i}(\mathbf{0})=0$. Finding a solution to this system of partial differential inequalities is equivalent to finding a Lyapunov function as in (\ref{eqnLyapSwitch}).\par
In Section \ref{secRKHS} we define a reproducing kernel Hilbert space $H$. We choose an approximating function $v$ of the function $V$ (\ref{eqnDiV}) to belong to such a space. We also introduce some properties of $H$ that are used later in Section \ref{secConvergence} and in Algorithm \ref{algMain}. In Section \ref{secDiscret} we discretise problem (\ref{eqnDiV}) by selecting collocation points. This gives us a minimisation problem which we solve to find an approximating function $v$. In Section \ref{secConvergence} we show that as the fill distance (\ref{eqnthmConvergence3}) decreases the approximating function $v$ converges strongly to $V$. This result is contained in Theorem \ref{thmConvergence}. An assumption on the fill distance in this theorem is addressed in Section \ref{secFillDistance}, where we present a way of choosing the collocation points under which the assumption holds.

\subsection{Reproducing kernel Hilbert spaces}
\label{secRKHS}

The solution $v$ to the minimisation problem (\ref{eqnMinProbi}), see Section \ref{secDiscret}, belongs to a reproducing kernel Hilbert space $H$. We recall various properties of these spaces within this section, as well as specifying the reproducing kernel $\Phi:\bR^{d}\times\bR^{d}\rightarrow\bR$ we will be using in the construction of the function $v$ in Section \ref{secQuadProg}. For more information on reproducing kernel Hilbert spaces, see \cite{wendland2004scattered}.
\begin{defn}{\rm (Reproducing Kernel Hilbert Space (RKHS))}
    \label{defnRKHS}
    Let $\emptyset\neq\Omega\subseteq\bR^{d}$. A Hilbert space $H=H(\Omega)$ of continuous functions $f:\Omega\rightarrow\bR$ with inner product $\langle\cdot,\cdot\rangle_{H}$ is called a reproducing kernel Hilbert space, abbreviated RKHS, if there is a function $\Phi:\Omega\times\Omega\rightarrow\bR$ with
    \begin{enumerate}
        \item $\Phi(\cdot,\mathbf{x})\in{H}$ for all ${\mathbf{x}\in\Omega}$
        \item $g(\mathbf{x})=\langle{g(\cdot),\Phi(\cdot,\mathbf{x})}\rangle_{H}$ for all $g\in{H}$ and all $\mathbf{x}\in\Omega$.
    \end{enumerate}
    The function $\Phi$ is called the reproducing kernel of $H$. The reproducing kernel is called positive definite if for any pair of pairwise distinct points $\{x_{1},...,x_{N}\}\subseteq\Omega$ the matrix $(\Phi(x_{i},x_{j}))_{i,j=1,...,N}$ is positive definite.
\end{defn}

As in \cite{giesl2021minimization}, the reproducing kernel Hilbert spaces that we consider are Sobolev spaces $H=H^{\alpha}(\Omega)$. We use the definition of the Sobolev space $W^{k,p}(\Omega)$ found in \cite{yosida2012functional} and set $H^{\alpha}(\Omega)=W^{\alpha,2}(\Omega)$ so that for $\alpha\in\bN_{0}$
\begin{equation}
    \label{eqnHalpha}
    H^{\alpha}(\Omega):=\{u\in{L_{2}}(\Omega):|\beta|\leq\alpha,D^{\beta}u\in{L_{2}}(\Omega)\}
\end{equation}
with scalar product $\displaystyle \langle{u,v}\rangle=\sum_{|\beta|\leq\alpha}\int_{\Omega}D^{\beta}u(\mathbf{x})D^{\beta}v(\mathbf{x})d\mathbf{x}$. If $\alpha$ is not an integer, then we use operator interpolation theory to define the space and norm, for more details see \cite{adams2003sobolev} and \cite{MR2373954}.\par
The following two lemmas will be used multiple times in Section \ref{secConvergence}, in the proof of both Theorem \ref{thmConvergence} and Lemma \ref{lemEquivOptimProb}. 

\begin{lemma}
    \label{lemBddLipsRKHS}
    If $\Omega\subseteq\bR^{d}$ is a bounded domain with Lipschitz boundary, then the space $H^{\alpha}(\Omega)$ for $\alpha>d/2$ is a RKHS.
\end{lemma}
\begin{proof}
    If $\Omega\subseteq\bR^{d}$ is a bounded domain with Lipschitz boundary, then there exists a continuous linear extension operator \cite{wendland2004scattered} $E:H^{\alpha}(\Omega)\rightarrow{H^{\alpha}(\bR^{d})}$. This has the properties $Eu\big|_{\Omega}=u$ and $\|Eu\|_{H^{\alpha}(\bR^{d})}\leq{C\|u\|_{H^{\alpha}(\Omega)}}$ for all $u\in{H^{\alpha}(\Omega)}$ with a fixed constant $C>0$. As stated in \cite{giesl2021minimization}, a consequence of the existence of the extension operator is the Sobolev embedding theorem for $H^{\alpha}(\Omega)$, which tells us that $H^{\alpha}(\Omega)\subseteq{C(\Omega)\cap{L_{\infty}(\Omega)}}$ if $\alpha>d/2$. As this is a continuous embedding, we have that $\|f\|_{C(\Omega)\cap{L_{\infty}(\Omega)}}\leq{C\|f\|_{H^{\alpha}(\Omega)}}$ for some $f\in{H^{\alpha}(\Omega)}$. The first norm can be rewritten as follows $\|f\|_{C(\Omega)\cap{L_{\infty}(\Omega)}}=\max_{x\in\Omega}|f(x)|$, which tells us that
    \begin{equation}
        \label{eqnlemprfBddLipsRKHS1}
        |f(x)|\leq\max_{x\in\Omega}|f(x)|\leq{C\|f\|_{H^{\alpha}(\Omega)}}\Rightarrow|\delta_{x}(f)|\leq\max_{x\in\Omega}|\delta_{x}(f)|\leq{C\|f\|_{H^{\alpha}(\Omega)}}.\nonumber
    \end{equation}
    This shows continuity of the point evaluations $\delta_{x}:H^{\alpha}(\Omega)\rightarrow\bR$, so our space $H^{\alpha}(\Omega)$ is a RKHS \cite{wendland2004scattered}.
\end{proof}

The explicit form of the reproducing kernel of $H^{\alpha}(\Omega)$ is usually unknown. To find a suitable substitute we use the following Lemma \ref{lemEquivNormsFourier} which can be found in \cite{giesl2021minimization} and was developed with results from \cite{wendland2004scattered}. 
The method of finding a suitable reproducing kernel for $H^{\alpha}(\Omega)$ can be summarised as follows: choose a reproducing kernel $K_{\alpha}(\mathbf{x},\mathbf{y})$ that delivers an equivalent norm to the standard norm on $H^{\alpha}(\bR^{d})$, and restrict it to find a reproducing kernel $k_{\alpha}(\mathbf{x},\mathbf{y})$ for $H^{\alpha}(\Omega)$. A reproducing kernel $K_{\alpha}$ that delivers this equivalent norm can be chosen to be a translation-invariant function $\Phi_{\alpha}(\mathbf{x}-\mathbf{y})$ with certain properties on its Fourier transform. $\Phi_{\alpha}$ can be chosen to be radial and, in particular, a Wendland function (see \cite{giesl2021minimization}), which we discuss further in (\ref{eqnWendlandDef}).

\begin{lemma}
    \label{lemEquivNormsFourier}
    Assume $\Phi_{\alpha}\in{L_{1}(\bR^{d})}\cap{C(\bR^{d})}$ has a Fourier transform $\hat{\Phi}_{\alpha}$ satisfying
    \begin{equation}
        \label{eqnlemEquivNormsFourier1}
        c_{1}(1+\|w\|_{2}^{2})^{-\alpha}\leq{\hat{\Phi}_{\alpha}(w)}\leq{c_{2}(1+\|w\|_{2}^{2})^{-\alpha}}\nonumber
    \end{equation}
    where $c_{1},c_{2}>0$ are constants, $w\in\bR^{d}$ and $\alpha>d/2$. Let $\Omega\subseteq\bR^{d}$ be a bounded domain with a Lipschitz boundary. Let $k_{\alpha}:\Omega\times\Omega\rightarrow\bR$ be defined by $k_{\alpha}(\mathbf{x},\mathbf{y})=\Phi_{\alpha}(\mathbf{x}-\mathbf{y})$, $\mathbf{x},\mathbf{y}\in\Omega$. Then there exists an inner product $\langle\cdot,\cdot\rangle_{k_{\alpha}}:H^{\alpha}(\Omega)\times{H^{\alpha}(\Omega)}\rightarrow\bR$ on $H^{\alpha}(\Omega)$ such that $k_{\alpha}$ is the reproducing kernel of $H^{\alpha}(\Omega)$ with respect to this inner product. The norm $\|\cdot\|_{k_{\alpha}}$ induced by this inner product is equivalent to the standard norm on $H^{\alpha}(\Omega)$, i.e. there are constants $C_{1},C_{2}>0$ such that
    \begin{equation}
        \label{eqnlemEquivNormsFourier2}
        C_{1}\|u\|_{k_{\alpha}}\leq\|u\|_{H^{\alpha}(\Omega)}\leq{C_{2}\|u\|_{k_{\alpha}}}\nonumber
    \end{equation}
\end{lemma}
\begin{proof}
    For $\alpha\in\bN$, this is Corollary 10.48 from \cite{wendland2004scattered}. For real $\alpha>d/2$ this then follows by interpolation theory \cite{giesl2021minimization}.
\end{proof}

In the construction of the function $v\in{H}$ we are using meshfree collocation. This method has been applied to the construction of Lyapunov functions for autonomous ODEs in \cite{giesl2007meshless} and \cite{giesl2018construction}. Following the methods explained here we assume that the reproducing kernel $\Phi$ of the Hilbert space $H$ is given by a radial basis function $\Psi_{0}$ through $\Phi(\mathbf{x},\mathbf{y})=\Psi_{0}(\|\mathbf{x}-\mathbf{y}\|_{2})$.\par
For this radial basis function we use Wendland functions, which are a family of functions introduced in \cite{wendland1995piecewise} and \cite{wendland1998error}. Wendland functions are positive definite with compact support, and are polynomials on their support. We define $\Psi_{0}(r):=\phi_{\ell,k}(cr)$, where $\phi_{\ell,k}(cr)$ is a Wendland function with $c>0$, $k\in\bN$ a smoothness parameter and $\ell=\lfloor\frac{d}{2}\rfloor+k+1$. The reproducing kernel Hilbert space that corresponds to this Wendland function is norm-equivalent to the Sobolev space $H^{k+(d+1)/2}(\bR^{d})$. Wendland functions are defined by recursion: for $\ell\in\bN$, $k\in\bN_{0}$
\begin{equation}
    \begin{split}
        \label{eqnWendlandDef}
        \phi_{\ell,0}(r) & =(1-r)_{+}^{\ell}\\
        \phi_{\ell,k+1}(r) & =\int_{r}^{1}t\phi_{\ell,k}(t)dt
    \end{split}
\end{equation}
for $r\in\bR_{0}^{+}$, where $y_{+}=y$ if $y\geq{0}$ and $y_{+}=0$ otherwise.\par
Using recursion, define $\Psi_{j}(r)=\displaystyle\frac{1}{r}\frac{d\Psi_{j-1}}{dr}(r)$ for $j=1,2$ and $r>0$. Note that $\lim_{r\rightarrow{0}}\Psi_{j}(r)$ exists if the smoothness parameter $k$ of the Wendland functions is sufficiently large.

\subsection{Discretised problem}
\label{secDiscret}

We now discretise our problem. Fix collocation points
\begin{equation}
    \label{eqnCollPtsOmega}
    \mathbf{z}^{1},\dots,\mathbf{z}^{N}\in\Omega\text{ where }Z:=\{\mathbf{z}^{1},\dots,\mathbf{z}^{N}\}
\end{equation}
where $N$ is the total number of collocation points, and take subsets of these points
\begin{equation}
    \label{eqnCollPtsi}
    X_{i}:=Z\cap{C_{i}}=\{\mathbf{x}_{i}^{1},\dots,\mathbf{x}_{i}^{\Li}\}
\end{equation}
for $i=1,\dots,k$, where $0\leq\Li\leq{N}$ is the number of collocation points in the subset $C_{i}$. We have that $X_{i}\subseteq{Z}$ for all $i\in{I}$, and
\begin{equation}
    \label{eqnZCollUnionXColl}
    Z=\bigcup_{i\in{I}}X_{i}.\nonumber
\end{equation}
The sets $X_{i}$ are not necessarily distinct as the sets $C_{i}$ can overlap.\par
We can now define
\begin{equation}
\label{eqnValbi}
    b_{i}^{1}=b_{i}(\mathbf{x}_{i}^{1}),\dots,b_{i}^{\Li}=b_{i}(\mathbf{x}_{i}^{\Li})\in\bR,
\end{equation}
and seek to find a function $v\in{H}$ that satisfies
\begin{equation}
    \label{eqnDiscretProbi}
    D_{i}v(\xxil)\leq\bil\text{ for all }\ell=1,\dots,\Li
\end{equation}
for $i=1,\dots,k$.\par
Another way to express (\ref{eqnDiscretProbi}) would be to find a function $v\in{H}$ using the information $\bil\in\bR$ generated by the functionals $\llil\in{H}^{*}$, i.e. $\llil(v)\leq\bil$ for all $i=1,\dots,k$ and $\ell=1,\dots,\Li$. We define these functionals as the orbital derivative with respect to $\mathbf{f}_{i}$ (\ref{eqnDi}) applied at a particular collocation point $\xxil$, i.e.
\begin{equation}
    \label{eqnllil}
    (\llil)v(\cdot):=(\delta_{\xxil}\circ{D_{i}})v(\cdot)=\nabla{v}(\xxil)\cdot\mathbf{f}_{i}(\xxil)
\end{equation}
for all $i=1,\dots,k$ and $\ell=1,\dots,\Li$, with $\delta$ being Dirac's delta distribution.\par
The optimal reconstruction of $v$ is then given by the solution of the problem
\begin{equation}
    \label{eqnMinProbi}
    \min\{\|v\|_{H}:v\in{H}\text{, }\llil(v)\leq\bil\text{, for all }i=1,\dots,k\text{ and }\ell=1,\dots,\Li\}.
\end{equation}
We now state the following lemma which is similar to one in \cite{giesl2018construction}, which shows that this minimisation problem has at most one solution.
\begin{lemma}
    \label{lemMinProbOneSol}
    The problem (\ref{eqnMinProbi}) has at most one solution.
\end{lemma}
\begin{proof}
    Assume $s\in{H}$ is a minimiser and $v\in{H}$ is a solution of
    \begin{equation}
        \label{eqnlemprfMinProbOneSol1}
        \llil(v)\leq\bil\text{, for all }i=1,\dots,k\text{ and }\ell=1,\dots,\Li.
    \end{equation}
    Then $\langle{s,v-s}\rangle_{H}\geq{0}$. We assume in contradiction to this that $\langle{s,v-s}\rangle_{H}<{0}$. For $\alpha\in[0,1]$, $t=\alpha{v}+(1-\alpha)s$ satisfies (\ref{eqnlemprfMinProbOneSol1}) and
    \begin{equation}
        \label{eqnlemprfMinProbOneSol2}
        \|t\|_{H}^{2}=\|s+\alpha(v-s)\|^{2}_{H}=\|s\|_{H}^{2}+2\alpha\langle{s,v-s}\rangle_{H}+\alpha^{2}\|v-s\|_{H}^{2}<\|s\|_{H}^{2}\nonumber
    \end{equation}
    for a suitable choice of $\alpha$. This is a contradiction to $s$ being a minimiser.\par
    Now let $s_{1},s_{2}\in{H}$ be minimisers. This implies that both $\langle{s_{1},s_{2}-s_{1}}\rangle_{H}$ and $\langle{s_{2},s_{1}-s_{2}}\rangle_{H}$ are greater than or equal to zero, giving
    \begin{equation}
        \label{eqnlemprfMinProbOneSol3}
        0\leq\|s_{1}-s_{2}\|_{H}^{2}=\langle{s_{1},s_{1}-s_{2}}\rangle_{H}-\langle{s_{2},s_{1}-s_{2}}\rangle_{H}\leq{0}\nonumber
    \end{equation}
    which shows $s_{1}=s_{2}.$
\end{proof}

\subsection{Convergence}
\label{secConvergence}

In this section we show that as the fill distance (\ref{eqnthmConvergence3}) decreases, the solution $v$ converges strongly to the solution of the original system of partial differential inequalities (\ref{eqnDiV}). We do this by adapting a theorem from \cite{giesl2021minimization} for autonomous ODEs.

\begin{theorem}
    \label{thmConvergence}
    Let $\Omega\subseteq\bR^{d}$ be a bounded domain with Lipschitz boundary. Let $C_{i}\subseteq{\Omega}$ for $i=1,\dots,k$ be such that $\mathbf{f}_{i}$ (\ref{eqnfis}) is the restriction to $C_{i}$ of a $C^{\alpha-1}$ vector field defined on an open neighbourhood of $C_{i}$ where $\alpha>d/2+2$. Let $D_{i}$ be linear differential operators as defined in (\ref{eqnDi}). 
   Let $b_{i}:C_{i}\rightarrow\bR$ be continuous functions and let $H=H^{\alpha}(\Omega)$ (\ref{eqnHalpha}) with norm given by an appropriate reproducing kernel. Consider the optimisation problem for $v\in{H}$
    \begin{equation}
        \label{eqnthmConvergence2}
        \begin{cases}
            \text{\rm{minimise }}\|v\|_{H}\\
            \text{\rm{subject to }}D_{i}v(\xxi)\leq{b_{i}}(\xxi)\text{, }\forall{i=1,\dots,k}\text{\rm{ and }}\xxi\in{C_{i}}
        \end{cases}
    \end{equation}
    and the sequence of optimisation problems:\\ 
    For $n\in\bN$ we choose points $Z_{n}:=\{\mathbf{z}^{1,n},\dots,\mathbf{z}^{N_{n},n}\}$. For $i=1,\dots,k$ the sets $\Xin$ are defined as $\Xin:=Z_{n}\cap{C_{i}}=\{\mathbf{x}_{i}^{1,n},\dots,\mathbf{x}_{i}^{L_{i,n},n}\}$, i.e. $\xxiln$ are elements of $\Xin$ for $\ell=1,\dots,\Lin$. We also assume that for fixed $i\in{I}$ the fill distance
    \begin{equation}
        \label{eqnthmConvergence3}
        h_{\Xin,C_{i}}:=\sup_{\xxi\in{C_{i}}}\min_{\ell=1,\dots,\Lin}\|\xxi-\xxiln\|_{2}\text{ satisfies }\lim_{n\rightarrow\infty}h_{\Xin,C_{i}}=0.
    \end{equation}
    Furthermore, assume that there exists a $V_{0}\in{H}$ that satisfies the constraints of (\ref{eqnthmConvergence2}). For each fixed $n\in\bN$ let $v^{\varepsilon}_{n}$ satisfy the constraints of the minimisation problem for $v\in{H}$
    \begin{equation}
        \label{eqnthmConvergence4}
        \begin{cases}
            \text{\rm{minimise }}\|v\|_{H}\\
            \text{\rm{subject to }}D_{i}v(\xxiln)\leq{b_{i}}(\xxiln)\text{, }i=1,\dots,k\text{, }\ell=1,\dots,\Lin
        \end{cases}
    \end{equation}
    with $\|v_{n}^{\varepsilon}\|_{H}\leq\frac{\varepsilon}{n}+m_{n}$ for some $\varepsilon>0$ where
    \begin{equation}
        \label{eqnthmConvergence5}
        m_{n}:=\inf_{v}\{\|v\|_{H}:v\text{ satisfies the constraints of (\ref{eqnthmConvergence4})}\}.
    \end{equation}
    Then the optimisation problem (\ref{eqnthmConvergence2}) has a unique solution $v$ and the functions $v_{n}^{\varepsilon}$ converge strongly in $H$ to $v$ as $n\rightarrow\infty$.
\end{theorem}

Before we prove this theorem we prove a lemma that shows how the minimisation problem (\ref{eqnthmConvergence4}) is related to our original problem (\ref{eqnMinProbi}).
\begin{lemma}
    \label{lemEquivOptimProb}
    For fixed $n\in\bN$, the optimisation problem (\ref{eqnthmConvergence4}) is the optimisation problem (\ref{eqnMinProbi}).
\end{lemma}
\begin{proof}
    We first start by seeing that as $\alpha>d/2$ and because of $\Omega$ having a Lipschitz boundary, $H:=H^{\alpha}(\Omega)$ is indeed a RKHS, as we have seen in Lemma \ref{lemBddLipsRKHS}. Now we can use Lemma \ref{lemEquivNormsFourier} and choose an inner product on $H^{\alpha}(\Omega)$ and a reproducing kernel $k_{\alpha}$ given by a positive definite function $\Phi_{\alpha}:\bR^{d}\rightarrow\bR$ as $k_{\alpha}(\mathbf{x},\mathbf{y})=\Phi_{\alpha}(\mathbf{x}-\mathbf{y})$, $\mathbf{x},\mathbf{y}\in\Omega$. The norm induced by this inner product is equivalent to the standard norm on $H^{\alpha}(\Omega)$, and we use this inner product and induced norm as the inner product and norm on $H^{\alpha}(\Omega)$, denoting $\langle\cdot,\cdot\rangle_{H}=\langle\cdot,\cdot\rangle_{H^{\alpha}(\Omega)}=\langle\cdot,\cdot\rangle_{k_{\alpha}}$. As $D_{i}$ maps $H^{\alpha}(\Omega)$ continuously to $H^{\alpha-1}(\Omega)$ and as $\alpha>d/2+1$ implies $H^{\alpha-1}(\Omega)\subseteq{C}(\Omega)\cap{L}_{\infty}(\Omega)$ by the Sobolev embedding theorem \cite{adams2003sobolev}, the functionals $\llil=\delta_{\xxiln}\circ{D_{i}}$ belong to $H^{*}$ as they are bounded linear operators. Hence, for a fixed $n\in\bN$, problem (\ref{eqnthmConvergence4}) is indeed (\ref{eqnMinProbi}) with these $\llil$, $\biln=b_{i}(\xxiln)$, $i=1,\dots,k$ and $\ell=1,\dots,\Lin$ where $\Li=\Lin$.
\end{proof}
Now we provide the proof of the convergence theorem.
\begin{proof}[Proof of Theorem \ref{thmConvergence}]
    Fix $\varepsilon>0$. We show in the following steps that the sequence $(v_{n}^{\varepsilon})_{n\in\bN}$ of solutions of (\ref{eqnthmConvergence4}) converges strongly to an element $v\in{H}$ which is the unique solution of (\ref{eqnthmConvergence2}). Due to the result of Lemma \ref{lemEquivOptimProb}, this implies that if the fill distance (\ref{eqnthmConvergence3}) decreases, the solution of the minimisation problem (\ref{eqnMinProbi}) converges strongly to a solution of the system of partial differential inequalities (\ref{eqnDiV}).\par
    The proof is structured as follows: in Step 1 we show that a subsequence of $(v_{n}^{\varepsilon})_{n\in\bN}$ weakly converges to a function $v\in{H}$, in Step 2 we show that $v$ satisfies the constraints of (\ref{eqnthmConvergence2}), in Step 3 using both previous steps we can now show that the subsequence strongly converges to $v$, and finally in Step 4 we show that $v$ is a unique minimiser. We then finish the proof by showing that our original sequence must converge strongly to $v$. We use the notation $\lliln:=\delta_{\xxiln}\circ{D}_{i}$ and $\biln=b_{i}(\xxiln)$.\medskip\\
    \textbf{Step 1:} We have a $V_{0}\in{H}^{\alpha}(\Omega)$ that satisfies the constraints of (\ref{eqnthmConvergence2}), which means that 
    \begin{equation}
        \label{eqnthmprfConvergence1}
        D_{i}V_{0}(\xxi)\leq{b_{i}}(\xxi)\text{, }\forall{i=1,\dots,k}\text{ and }\xxi\in{C_{i}}.\nonumber
    \end{equation}
    This means that
    \begin{equation}
        \label{eqnthmprfConvergence2}
        \lliln(V_{0})=D_{i}V_{0}(\xxiln)\leq{b_{i}(\xxiln)}=\biln.\nonumber
    \end{equation}
    Thus $V_{0}$ satisfies the constraints of (\ref{eqnthmConvergence4}), which means by assuming that there is a feasible solution to (\ref{eqnthmConvergence2}) we also have the assumption that there is a feasible solution to (\ref{eqnthmConvergence4}) for each $n\in\bN$. Therefore, using the definition of $m_{n}$ in (\ref{eqnthmConvergence5}), there exists $v_{n}^{\varepsilon}$ that satisfies the constraints of (\ref{eqnthmConvergence4}) with $\|v_{n}^{\varepsilon}\|_{H}\leq\frac{\varepsilon}{n}+m_{n}$.\par
    Hence we have
    \begin{equation}
        \label{eqnthmprfConvergence3}
        \|v_{n}^{\varepsilon}\|_{H}\leq\|V_{0}\|_{H}+\frac{\varepsilon}{n}.
    \end{equation}
    In particular, 
    \begin{equation}
        \label{eqnthmprfConvergence3.1}
        \|v_{n}^{\varepsilon}\|_{H}\leq\|V_{0}\|_{H}+\varepsilon,
    \end{equation} 
    thus $\|v_{n}^{\varepsilon}\|_{H}$ is bounded for all $n\in\bN$. As bounded sets in Hilbert spaces are relatively compact in the weak topology, there is a subsequence of $(v_{n}^{\varepsilon})_{n\in\bN}$ which we also denote by $(v_{n}^{\varepsilon})_{n\in\bN}$, that weakly converges to a function $v\in{H}$. From
    \begin{equation}
        \label{eqnthmprfConvergence4}
        \|v\|_{H}^{2}=\langle{v,v-v_{n}^{\varepsilon}}\rangle_{H}+\langle{v,v_{n}^{\varepsilon}}\rangle_{H}\leq\langle{v,v-v_{n}^{\varepsilon}}\rangle_{H}+\|v\|_{H}\|v_{n}^{\varepsilon}\|_{H}\nonumber
    \end{equation}
    it follows using (\ref{eqnthmprfConvergence3}) that
    \begin{equation}
        \label{eqnthmprfConvergence5}
        \|v\|_{H}\leq\liminf_{n\rightarrow\infty}\|v_{n}^{\varepsilon}\|_{H}\leq\|V_{0}\|_{H}:=C_{0}.
    \end{equation}\medskip\\
    \textbf{Step 2:} Now we use the kernel representation to show that $D_{i}v(\xxi)\leq{b_{i}}(\xxi)$ for all $i=1,\dots,k$ and $\xxi\in{C_{i}}$, where $v$ is the weak limit of the subsequence $(v_{n}^{\varepsilon})_{n\in\bN}$, i.e. that $v$ satisfies the constraints of (\ref{eqnthmConvergence2}).\par 
    We can define $\lli:=\delta_{\xxi}\circ{D_{i}}\in{H}^{*}$ where $(\lli)^{\mathbf{y}}k_{\alpha}(\cdot,\mathbf{y})$ is the Riesz representative of $\lli$. This gives that
    \begin{equation}
        \label{eqnthmprfConvergence6}
        |D_{i}v(\xxi)-D_{i}v_{n}^{\varepsilon}(\xxi)|=|\lli(v-v_{n}^{\varepsilon})|=\langle{v-v_{n}^{\varepsilon}},(\lli)^{\mathbf{y}}k_{\alpha}(\cdot,\mathbf{y})\rangle_{H}\rightarrow{0}
    \end{equation}
    as $v_{n}^{\varepsilon}$ converges weakly to $v$.\par
    Next we can see from Lemma \ref{lemBddLipsRKHS} that as $\alpha-1>d/2$ and $\Omega$ has Lipschitz boundary, $H^{\alpha-1}(\Omega)$ is also a RKHS. We can again use Lemma \ref{lemEquivNormsFourier} to give us a reproducing kernel $k_{\alpha-1}(\mathbf{x},\mathbf{y})=\Phi_{\alpha-1}(\mathbf{x}-\mathbf{y})$ where $\Phi_{\alpha-1}:\bR^{d}\rightarrow\bR$. Since $\alpha-1>d/2+1$ we can also use the Sobolev embedding theorem \cite{adams2003sobolev} to give $H^{\alpha-1}(\bR^{d})\subseteq{W_{\infty}^{1}}(\bR^{d})\cap{C^{1}}(\bR^{d})$ which implies there exists an $M>0$ such that 
    \begin{equation}
        \label{eqnthmprfConvergence7}
        \|\nabla\Phi_{\alpha-1}(\xi)\|_{2}\leq{M}
    \end{equation}
    for all $\xi\in\bR^{d}$. Indeed, using that $\Phi_{\alpha-1}\in{H^{\alpha-1}(\bR^{d})}$, the Sobolev embedding theorem gives us the inequality $\|\Phi_{\alpha-1}\|_{W_{\infty}^{1}(\bR^{d})\cap{C^{1}(\bR^{d})}}\leq{C}\|\Phi_{\alpha-1}\|_{H^{\alpha-1}(\bR^{d})}$. We then find that
    \begin{equation}
        \label{eqnthmprfConvergence8}
        \|\Phi_{\alpha-1}\|_{W_{\infty}^{1}(\bR^{d})\cap{C^{1}(\bR^{d})}}=\max_{0\leq|\beta|\leq{1}}\|D^{\beta}\Phi_{\alpha-1}(\xi)\|_{\infty}\geq\|\nabla\Phi_{\alpha-1}(\xi)\|_{\infty}\geq\frac{1}{\sqrt{d}}\|\nabla\Phi_{\alpha-1}(\xi)\|_{2}\nonumber
    \end{equation}
    which gives us the original inequality for all $\xi\in\bR^{d}$.\par
    Now we can take $\xxi,\yyi\in{C_{i}}$ and a $\xi\in\bR^{d}$ on the line between $0$ and $\xxi-\yyi$ such that, using that $D_{i}:H^{\alpha}(\Omega)\rightarrow{H^{\alpha-1}}(\Omega)$ is a bounded operator with constant $c_{0}$,
    \begin{eqnarray}
        |D_{i}v_{n}^{\varepsilon}(\xxi)-D_{i}v_{n}^{\varepsilon}(\yyi)| & = & |\langle{D_{i}}v_{n}^{\varepsilon},k_{\alpha-1}(\cdot,\xxi)-k_{\alpha-1}(\cdot,\yyi)\rangle_{H^{\alpha-1}(\Omega)}|\nonumber\\
        & \mathrel{\overset{\makebox[0pt]{\mbox{\normalfont\tiny\sffamily (C-S)}}}{\leq}} & \|D_{i}v_{n}^{\varepsilon}\|_{H^{\alpha-1}(\Omega)}\|{k_{\alpha-1}(\cdot,\xxi)-k_{\alpha-1}(\cdot,\yyi)}\|_{H^{\alpha-1}(\Omega)}\nonumber\\
        & = & \|D_{i}v_{n}^{\varepsilon}\|_{H^{\alpha-1}(\Omega)}(k_{\alpha-1}(\xxi,\xxi)+k_{\alpha-1}(\yyi,\yyi)\nonumber\\
        & & \phantom{ccccccccccccccccccccccccccccc}-2k_{\alpha-1}(\xxi,\yyi))^{1/2}\nonumber\\
        & \leq & \sqrt{2}c_{0}\|v_{n}^{\varepsilon}\|_{H^{\alpha}(\Omega)}(\Phi_{\alpha-1}(0)-\Phi_{\alpha-1}(\xxi-\yyi))^{1/2}\nonumber\\
        & \mathrel{\overset{\makebox[0pt]{\mbox{\normalfont\tiny\sffamily (\ref{eqnthmprfConvergence3.1})}}}{\leq}} & \sqrt{2}c_{0}(C_{0}+\varepsilon)(\Phi_{\alpha-1}(0)-\Phi_{\alpha-1}(\xxi-\yyi))^{1/2}\nonumber\\
        & \mathrel{\overset{\makebox[0pt]{\mbox{\normalfont\tiny\sffamily (MVT)}}}{\leq}} & \sqrt{2}c_{0}(C_{0}+\varepsilon)\|\nabla\Phi_{\alpha-1}(\xi)\|_{2}^{1/2}\|\xxi-\yyi\|_{2}^{1/2}\nonumber\\
        \label{eqnthmprfConvergence9}
        & \mathrel{\overset{\makebox[0pt]{\mbox{\normalfont\tiny\sffamily (\ref{eqnthmprfConvergence7})}}}{\leq}} & C_{1}\|\xxi-\yyi\|_{2}^{1/2}
    \end{eqnarray}
    for all $n\in\bN$ where $C_{1}=\sqrt{2}c_{0}(C_{0}+\varepsilon)M^{1/2}$. The second equality is deduced using $\|\cdot\|_{H^{\alpha-1}(\Omega)}=\sqrt{\langle{\cdot,\cdot}\rangle_{H^{\alpha-1}(\Omega)}}$, the properties of inner products and the properties of $k_{\alpha-1}$ as a reproducing kernel as described in Definition \ref{defnRKHS}.\par
    We have set out to show that $D_{i}v(\xxi)\leq{b_{i}}(\xxi)$ for all $i=1,\dots,k$ and $\xxi\in{C_{i}}$, which is the same as showing that for all $\mu>0$ we have $D_{i}v(\xxi)-{b_{i}}(\xxi)<\mu$.\par
    We fix $\mu>0$. By (\ref{eqnthmprfConvergence6}) we know that there exists ${N_{1}}\in\bN$ such that for all ${n}\geq{N_{1}}$
    \begin{equation}
        \label{eqnthmprfConvergence10}
        |D_{i}v_{n}^{\varepsilon}(\xxi)-D_{i}v(\xxi)|<\frac{\mu}{3}.
    \end{equation}
    Since $b_{i}$ is continuous at $\xxi$ there exists a $\delta>0$ such that for all $\yyi\in{C_{i}}$ with
    \begin{equation}
        \label{eqnthmprfConvergence11}
        |\xxi-\yyi|<\delta\Rightarrow|b_{i}(\xxi)-b_{i}(\yyi)|<\frac{\mu}{3}.
    \end{equation}
    We also know by (\ref{eqnthmConvergence3}) that there exists $N_{2}\in\bN$ such that for all $n\geq{N_{2}}$ there exists a $\xxiln\in\Xin$ with
    \begin{equation}
        \label{eqnthmprfConvergence12}
        \|\xxi-\xxiln\|_{2}\leq\min\left(\frac{\mu^{2}}{9C_{1}^{2}},\delta\right).
    \end{equation}
    Therefore, for $n\geq\max(N_{1},N_{2})$ and using the statements (\ref{eqnthmprfConvergence10}), (\ref{eqnthmprfConvergence9}), (\ref{eqnthmprfConvergence12}) and (\ref{eqnthmprfConvergence11}) as well as $D_{i}v_{n}^{\varepsilon}(\xxiln)\leq{b_{i}(\xxiln)}$ we have that
    \begin{eqnarray}
        D_{i}v(\xxi)-{b_{i}}(\xxi) & = & (D_{i}v(\xxi)-D_{i}v_{n}^{\varepsilon}(\xxi))+(D_{i}v_{n}^{\varepsilon}(\xxi)-D_{i}v_{n}^{\varepsilon}(\xxiln))\nonumber\\
        & & +(D_{i}v_{n}^{\varepsilon}(\xxiln)-b_{i}(\xxiln))+(b_{i}(\xxiln)-b_{i}(\xxi))\nonumber\\
        & < & \frac{\mu}{3}+C_{1}\frac{\mu}{3C_{1}}+0+\frac{\mu}{3}=\mu.\nonumber
    \end{eqnarray}\medskip\\
    \textbf{Step 3:} Due to our findings in Step 2 that $v$ satisfies the constraints of the continuous problem (\ref{eqnthmConvergence2}), it satisfies them in particular at every collocation point $\xxiln$. Thus $v$ is feasible for the discretised problem (\ref{eqnthmConvergence4}), for every $n\in\bN$. Therefore, we have $m_{n}\leq\|v\|_{H}$ and thus
    \begin{equation}
        \label{eqnthmprfConvergence12.1}
        \|v_{n}^{\varepsilon}\|_{H}-\frac{\varepsilon}{n}\leq\|v\|_{H}
    \end{equation}
    for all $n\in\bN$. This implies that $\limsup_{n\rightarrow\infty}\|v_{n}^{\varepsilon}\|_{H}\leq\|v\|_{H}$ holds. Now using (\ref{eqnthmprfConvergence5}) it follows that
    \begin{equation}
        \label{eqnthmprfConvergence13}
        \lim_{n\rightarrow\infty}\|v_{n}^{\varepsilon}\|_{H}=\|v\|_{H}\nonumber
    \end{equation}
    and as
    \begin{equation}
        \label{eqnthmprfConvergence14}
        \|v-v_{n}^{\varepsilon}\|_{H}^{2}=\|v_{n}^{\varepsilon}\|_{H}^{2}-\|v\|_{H}^{2}+2\langle{v,v-v_{n}^{\varepsilon}}\rangle_{H}\nonumber
    \end{equation}
    it follows that $v_{n}^{\varepsilon}$ converges strongly to $v$.
    \medskip\\
    \textbf{Step 4:} It remains to show that $v$ is a unique minimiser. Firstly we show that $v$ is a minimiser. To do this we take a $V\in{H}$ that satisfies the constraints of (\ref{eqnthmConvergence2}), which implies that for every $n$, $V$ also satisfies the constraints of (\ref{eqnthmConvergence4}). Therefore we have that $\|v_{n}^{\varepsilon}\|_{H}-\varepsilon/n\leq\|V\|_{H}$ and, if we take $n\rightarrow\infty$, we have that $\|v_{n}^{\varepsilon}\|_{H}\rightarrow\|v\|_{H}$ due to the strong convergence, so $\|v\|_{H}\leq\|V\|_{H}$. This shows $v$ is a minimiser.\par
    To show it is a unique minimiser we assume $s\in{H}$ is a minimiser and $v\in{H}$ satisfies the constraints of (\ref{eqnthmConvergence2}) and show that
    \begin{equation}
        \label{eqnthmprfConvergence15}
        \langle{s,v-s}\rangle_{H}\geq{0}
    \end{equation}
    Indeed we assume $\langle{s,v-s}\rangle_{H}<0$. If we let $\alpha\in[0,1]$ we can see for $t=\alpha{v}+(1-\alpha)s$ (which satisfies the constraints) we have
    \begin{equation}
        \label{eqnthmprfConvergence16}
        \|t\|_{H}^{2}=\|s\|_{H}^{2}+2\alpha\langle{s,v-s}\rangle_{H}+\alpha^{2}\|v-s\|_{H}^{2}<\|s\|_{H}^{2}\nonumber
    \end{equation}
    for a suitable choice of $\alpha>0$. This contradicts $s$ being a minimiser.\par
    Now let $s_{1},s_{2}\in{H}$ be minimisers. By (\ref{eqnthmprfConvergence15}) that means we have
    \begin{equation}
        \label{eqnthmprfConvergence17}
        \langle{s_{1},s_{2}-s_{1}}\rangle_{H}\geq{0}\text{ and }\langle{s_{2},s_{1}-s_{2}}\rangle_{H}\geq{0}.\nonumber
    \end{equation}
    This implies
    \begin{equation}
        \label{eqnthmprfConvergence18}
        0\leq\|s_{1}-s_{2}\|_{H}^{2}=-\langle{s_{1},s_{2}-s_{1}}\rangle_{H}-\langle{s_{2},s_{1}-s_{2}}\rangle_{H}\leq{0}\nonumber
    \end{equation}
    so $s_{1}=s_{2}$.\par
    Hence every weakly convergent subsequence of the original sequence $(v_{n}^{\varepsilon})_{n\in\bN}$ converges strongly to the unique minimiser $v$ of (\ref{eqnthmConvergence2}). From this it follows that the original sequence $(v_{n}^{\varepsilon})_{n\in\bN}$ of functions that satisfy the constraints of (\ref{eqnthmConvergence4}) converges strongly to $v$. 
\end{proof}
Two assumptions are made for this convergence result to hold. One assumption is on the fill distance of the collocation points converging to zero, and sufficient conditions for this to hold are given in the following section. The other assumption is that there exists a function $V_{0}$ that satisfies the constraints of the minimisation problem (\ref{eqnthmConvergence2}). A sufficient condition for this is the existence of a Lyapunov function for the switched system (\ref{eqnfis}). We will explore sufficient conditions for the existence of a Lyapunov function for this system in future work.

\subsection{Fill distance assumption}
\label{secFillDistance}

In this section, we consider the assumption made in (\ref{eqnthmConvergence3}) that states that the fill distance tends to zero as the number of collocation points tends to infinity. One way of choosing the collocation points is setting $Z_{n}=2^{-n}\bZ^{d}\cap\Omega$ with $n\in\bN$. Then for all $i\in{I}$ we have a subset of the collocation points $\Xin=2^{-n}\bZ^{d}\cap{C_{i}}$. If we know that the sets $C_{i}$ have Lipschitz boundary and we choose the collocation points in this way, the fill distance converges as we want for Theorem \ref{thmConvergence}.\par
Any set that is bounded with Lipschitz boundary satisfies an interior cone condition, see Section 4.7 on page 67 of \cite{adams2003sobolev}. For every point $\mathbf{x}\in\Omega$ of a set $\Omega\subseteq{\bR^{d}}$ satisfying an interior cone condition there is a direction (given by unit vector $\xi(\mathbf{x})$) such that a cone with angle $\alpha$ and radius $r$ is contained within $\Omega$. This has been illustrated in Figure \ref{figSetConeCond}, and a formal definition is given below, see \cite{wendland2004scattered}.

\begin{defn}{\rm (Interior Cone Condition)}
    \label{defnConeCond}
    A set $\Omega\subseteq\bR^{d}$ is said to satisfy an interior cone condition if there exists an angle $\alpha\in(0,\pi/2)$ and a radius $r>0$ such that for every $\mathbf{x}\in\Omega$ a unit vector $\xi(\mathbf{x})$ exists such that the cone
    \begin{equation}
        \label{eqndefnConeCond1}
        C(\mathbf{x},\xi(\mathbf{x}),\alpha,r):=\{\mathbf{x}+\lambda\mathbf{y}:\mathbf{y}\in\bR^{d}\text{, }\|\mathbf{y}\|_{2}=1\text{, }\mathbf{y}^{T}\xi(\mathbf{x})\geq\cos(\alpha)\text{, }\lambda\in[0,r]\}
    \end{equation}
    is contained in $\Omega$.
\end{defn}

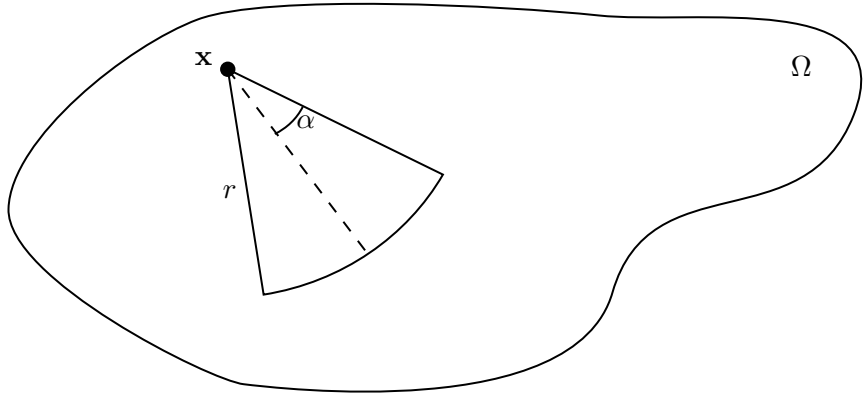
\begin{figure}[H]
    \centering
\tikzset{every picture/.style={line width=0.75pt}} 

\begin{tikzpicture}[x=0.75pt,y=0.75pt,yscale=-1,xscale=1]

\draw   (279.95,119.34) .. controls (262.03,150.08) and (229.62,173.25) .. (189.99,179.55) -- (172,66.5) -- cycle ;
\draw   (358,39.5) .. controls (406,44.5) and (511,25.5) .. (486,89.5) .. controls (461,153.5) and (384,112.5) .. (365,178.5) .. controls (346,244.5) and (197,226.5) .. (180,224.5) .. controls (163,222.5) and (61,171.5) .. (62,136.5) .. controls (63,101.5) and (122,54.5) .. (157,41.5) .. controls (192,28.5) and (310,34.5) .. (358,39.5) -- cycle ;
\draw  [dash pattern={on 4.5pt off 4.5pt}]  (172,66.5) -- (241,157.5) ;
\draw  [draw opacity=0] (209.63,85.33) .. controls (206.57,91.37) and (201.63,95.98) .. (195.61,98.96) -- (172,66.5) -- cycle ; \draw   (209.63,85.33) .. controls (206.57,91.37) and (201.63,95.98) .. (195.61,98.96) ;  
\draw    (100,114) ;
\draw    (172,66.5) ;
\draw [shift={(172,66.5)}, rotate = 0] [color={rgb, 255:red, 0; green, 0; blue, 0 }  ][fill={rgb, 255:red, 0; green, 0; blue, 0 }  ][line width=0.75]      (0, 0) circle [x radius= 3.35, y radius= 3.35]   ;

\draw (453,58.4) node [anchor=north west][inner sep=0.75pt]    {$\Omega $};
\draw (205,88.4) node [anchor=north west][inner sep=0.75pt]    {$\alpha $};
\draw (168,123.4) node [anchor=north west][inner sep=0.75pt]    {$r$};
\draw (154,56.4) node [anchor=north west][inner sep=0.75pt]    {$\mathbf{x}$};

\end{tikzpicture}

    \caption{A set $\Omega\subseteq\bR^{2}$ satisfying the cone condition with the cone $C(\mathbf{x},\xi(\mathbf{x}),\alpha,r)$ pictured. Every point in $\Omega$ has a cone of the same angle $\alpha$ and radius $r$ contained fully within the set.}
    \label{figSetConeCond}
\end{figure}


See the following lemma, also from \cite{wendland2004scattered}, showing that there is a ball contained within this cone.

\begin{lemma}
    \label{lemConeCondBall}
    Suppose that $C=C(\mathbf{x},\xi,\alpha,r)$ is a cone defined as in (\ref{eqndefnConeCond1}). Then for every $h\leq{r/(1+\sin(\alpha))}$ the closed ball $B(\mathbf{y},h\sin(\alpha))$ with centre $\mathbf{y}=\mathbf{x}+h\xi$ and radius $h\sin(\alpha)$ is contained in $C(\mathbf{x},\xi,\alpha,r)$. In particular, if $\mathbf{z}$ is a point in this ball, then the whole line segment $\mathbf{x}+t(\mathbf{z}-\mathbf{x})/\|\mathbf{z}-\mathbf{x}\|_{2}$, $t\in[0,r]$, is contained in the cone.
\end{lemma}

We now prove the following lemma that we go on to use in the proof of Corollary \ref{corFillDistance}. There we detail sufficient conditions for the assumption (\ref{eqnthmConvergence3}) to hold.

\begin{lemma}
    \label{lemConeCondBound}
    Set the collocation points to be $Z_{n}=2^{-n}\bZ^{d}\cap\Omega$ and take subsets $\Xin=2^{-n}\bZ^{d}\cap{C_{i}}$. Fix $i\in{I}$ and assume $C_{i}$ satisfies an interior cone condition. Then for all $B>{0}$ there exists ${N}\in\bN_{0}$ such that for all $n>{N}$ and $\xxi\in{C_{i}}$ there exists a collocation point $\xxiln\in\Xin$ such that $\|\xxi-\xxiln\|_{2}\leq{B}$.
\end{lemma}
\begin{proof}
    Fix $i\in{I}$. The set $C_{i}$ satisfies an interior cone condition, as described in Definition \ref{defnConeCond}. For a point $\xxi\in{C_{i}}$ its respective cone is denoted $C(\xxi,\xi(\xxi),\alpha,r)$. Applying Lemma \ref{lemConeCondBall} we see that for every $h\leq{r/(1+\sin(\alpha))}$ the closed ball $B(\yyi,h\sin(\alpha))$ with centre $\yyi=\xxi+h\xi$ and radius $h\sin(\alpha)$ is contained within the cone $C(\xxi,\xi(\xxi),\alpha,r)$. An example of a cone and a ball for a point $\xxi\in{C_{i}}$ is shown in Figure \ref{figCone}.
    \begin{figure}[H]
    \centering

\tikzset{every picture/.style={line width=0.75pt}} 

\begin{tikzpicture}[x=0.6pt,y=0.6pt,yscale=-1,xscale=1]

\draw   (331.04,35.82) .. controls (343.13,66.21) and (349.85,99.32) .. (350.04,134.01) .. controls (350.22,167.14) and (344.42,198.91) .. (333.64,228.29) -- (80.38,135.46) -- cycle ;
\draw   (202.05,133.07) .. controls (202.05,96.09) and (232.03,66.12) .. (269.01,66.12) .. controls (305.98,66.12) and (335.96,96.09) .. (335.96,133.07) .. controls (335.96,170.05) and (305.98,200.02) .. (269.01,200.02) .. controls (232.03,200.02) and (202.05,170.05) .. (202.05,133.07) -- cycle ;
\draw  [dash pattern={on 4.5pt off 4.5pt}]  (80.38,135.46) -- (269.01,133.07) ;
\draw    (269.01,133.07) ;
\draw [shift={(269.01,133.07)}, rotate = 0] [color={rgb, 255:red, 0; green, 0; blue, 0 }  ][fill={rgb, 255:red, 0; green, 0; blue, 0 }  ][line width=0.75]      (0, 0) circle [x radius= 3.35, y radius= 3.35]   ;
\draw  [draw opacity=0] (126.59,116.58) .. controls (130.22,121.28) and (132.1,127.35) .. (131.39,133.73) .. controls (131.36,133.99) and (131.33,134.25) .. (131.29,134.5) -- (107.77,131.09) -- cycle ; \draw   (126.59,116.58) .. controls (130.22,121.28) and (132.1,127.35) .. (131.39,133.73) .. controls (131.36,133.99) and (131.33,134.25) .. (131.29,134.5) ;  
\draw    (269.21,131.08) -- (275.54,68.5) ;
\draw [shift={(275.74,66.51)}, rotate = 95.78] [color={rgb, 255:red, 0; green, 0; blue, 0 }  ][line width=0.75]    (10.93,-3.29) .. controls (6.95,-1.4) and (3.31,-0.3) .. (0,0) .. controls (3.31,0.3) and (6.95,1.4) .. (10.93,3.29)   ;
\draw [shift={(269.01,133.07)}, rotate = 275.78] [color={rgb, 255:red, 0; green, 0; blue, 0 }  ][line width=0.75]    (10.93,-3.29) .. controls (6.95,-1.4) and (3.31,-0.3) .. (0,0) .. controls (3.31,0.3) and (6.95,1.4) .. (10.93,3.29)   ;
\draw    (80.38,135.46) ;
\draw [shift={(80.38,135.46)}, rotate = 0] [color={rgb, 255:red, 0; green, 0; blue, 0 }  ][fill={rgb, 255:red, 0; green, 0; blue, 0 }  ][line width=0.75]      (0, 0) circle [x radius= 3.35, y radius= 3.35]   ;

\draw (56.98,124.07) node [anchor=north west][inner sep=0.75pt]    {$\mathbf{x}_{i}$};
\draw (262.99,139.17) node [anchor=north west][inner sep=0.75pt]    {$\mathbf{y}_{i}$};
\draw (115.23,121.11) node [anchor=north west][inner sep=0.75pt]    {$\alpha $};
\draw (186.54,116.36) node [anchor=north west][inner sep=0.75pt]    {$h$};
\draw (215.3,87.83) node [anchor=north west][inner sep=0.75pt]    {\small{$h\sin( \alpha )$}};

\end{tikzpicture}

    \caption{The cone $C(\xxi,\xi(\xxi),\alpha,r)$ containing the ball $B(\yyi,h\sin(\alpha))$ as described in Lemma \ref{lemConeCondBall}.}
    \label{figCone}
\end{figure}
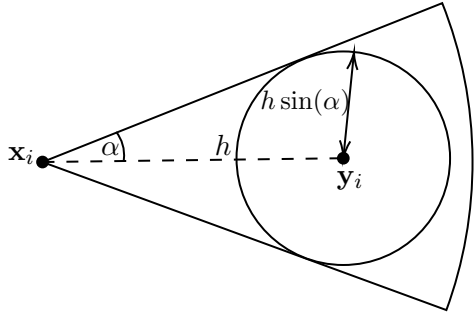
    Throughout this proof we use the following two choices of variables. Choose $h$ to be such that
    \begin{equation}
        \label{eqnlemConeCondBoundprf1}
        h\leq{\min\left(\frac{r}{1+\sin(\alpha)},\frac{B}{1+\sin(\alpha)}\right)}
    \end{equation}
    and choose $N\in\bN_{0}$ to be
    \begin{equation}
        \label{eqnlemConeCondBoundprf2}
        N:=\begin{cases}
            \log_{2}\left(\frac{\sqrt{d}}{2h\sin(\alpha)}\right) & \text{if }0<h\sin(\alpha)\leq\frac{\sqrt{d}}{2}\\
            0 & \text{if }h\sin(\alpha)>\frac{\sqrt{d}}{2}
        \end{cases}
    \end{equation}\par
    We prove this lemma in two steps. Firstly we show that for $n>{N}$ there exists a collocation point $\xxiln\in\Xin$ within the ball $B(\yyi,h\sin(\alpha))$. Then, using the restriction placed on $h$ in (\ref{eqnlemConeCondBoundprf1}), we show that the distance between this collocation point $\xxiln$ and an arbitrary point $\xxi\in{C_{i}}$ is bounded by $B$.\medskip\\
    \textbf{Step 1:} For $n>N$ and all $\xxi\in{C_{i}}$ there will exist a collocation point $\xxiln$ within the ball $B(\yyi,h\sin(\alpha))$ if
    \begin{equation}
        \label{eqnlemConeCondBoundprf2.1}
        \frac{1}{2^{n}}<\frac{2h\sin(\alpha)}{\sqrt{d}}.
    \end{equation}
    This is because the distance collocation points are from each other is $2^{-n}$, and the smallest distance between the collocation points $\Xin$ such that there could be no points within a ball $B(\mathbf{y}_{i},h\sin(\alpha))$ is $2h\sin(\alpha)/\sqrt{d}$. We show this to be the case in the following way: for $d=1$ this is clear as the distance between the collocation points is the same as the diameter of the ball. For $d\geq{2}$ we can show this using contradiction. Assume that the distance between the collocation points is $a$ with $a<2h\sin(\alpha)/\sqrt{d}$ and there are no collocation points within the ball. Let $\xxiln$ be a collocation point that is closest to the point $\yyi$. The distance between $\xxiln$ and $\yyi$ must be less than or equal to half of the distance between the collocation points $a$ multiplied by $\sqrt{d}$, as a unit hypercube's longest diagonal in $d$ dimensions is equal to $\sqrt{d}$. This implies that
\begin{equation}
    \label{eqnCompBallRad&CollPtDist}
    h\sin(\alpha)<\|\yyi-\xxiln\|_{2}\leq\frac{1}{2}a\sqrt{d}<h\sin(\alpha),\nonumber
\end{equation}
which is a contradiction. In Figure \ref{figBall} we show the point $\yyi\in\bR^{2}$ with a ball of radius $h\sin(\alpha)$ around it. It can be seen here that the smallest distance that is between the collocation points such that there are no points within the ball is $2h\sin(\alpha)/\sqrt{2}$.\\
    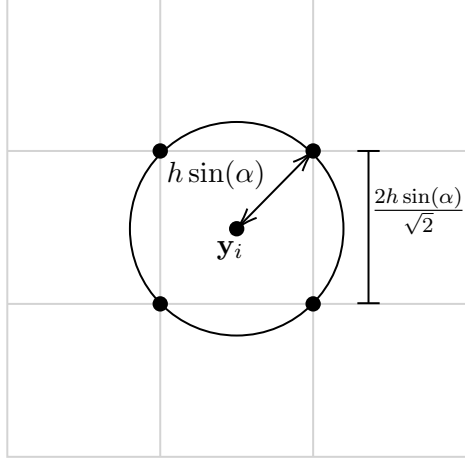
\begin{figure}[H]
        \centering

\tikzset{every picture/.style={line width=0.75pt}} 

\begin{tikzpicture}[x=0.75pt,y=0.75pt,yscale=-1,xscale=1]

\draw  [draw opacity=0] (14,13) -- (244.15,13) -- (244.15,243.15) -- (14,243.15) -- cycle ; \draw  [color={rgb, 255:red, 211; green, 211; blue, 211 }  ,draw opacity=1 ] (90.72,13) -- (90.72,243.15)(167.43,13) -- (167.43,243.15) ; \draw  [color={rgb, 255:red, 211; green, 211; blue, 211 }  ,draw opacity=1 ] (14,89.72) -- (244.15,89.72)(14,166.43) -- (244.15,166.43) ; \draw  [color={rgb, 255:red, 211; green, 211; blue, 211 }  ,draw opacity=1 ] (14,13) -- (244.15,13) -- (244.15,243.15) -- (14,243.15) -- cycle ;
\draw    (90.72,89.72) ;
\draw [shift={(90.72,89.72)}, rotate = 0] [color={rgb, 255:red, 0; green, 0; blue, 0 }  ][fill={rgb, 255:red, 0; green, 0; blue, 0 }  ][line width=0.75]      (0, 0) circle [x radius= 3.35, y radius= 3.35]   ;
\draw    (90.72,166.43) ;
\draw [shift={(90.72,166.43)}, rotate = 0] [color={rgb, 255:red, 0; green, 0; blue, 0 }  ][fill={rgb, 255:red, 0; green, 0; blue, 0 }  ][line width=0.75]      (0, 0) circle [x radius= 3.35, y radius= 3.35]   ;
\draw    (167.43,89.72) ;
\draw [shift={(167.43,89.72)}, rotate = 0] [color={rgb, 255:red, 0; green, 0; blue, 0 }  ][fill={rgb, 255:red, 0; green, 0; blue, 0 }  ][line width=0.75]      (0, 0) circle [x radius= 3.35, y radius= 3.35]   ;
\draw    (167.43,166.43) ;
\draw [shift={(167.43,166.43)}, rotate = 0] [color={rgb, 255:red, 0; green, 0; blue, 0 }  ][fill={rgb, 255:red, 0; green, 0; blue, 0 }  ][line width=0.75]      (0, 0) circle [x radius= 3.35, y radius= 3.35]   ;
\draw   (75.51,128.74) .. controls (75.51,99.15) and (99.49,75.17) .. (129.08,75.17) .. controls (158.66,75.17) and (182.65,99.15) .. (182.65,128.74) .. controls (182.65,158.32) and (158.66,182.31) .. (129.08,182.31) .. controls (99.49,182.31) and (75.51,158.32) .. (75.51,128.74) -- cycle ;
\draw    (129.08,128.74) ;
\draw [shift={(129.08,128.74)}, rotate = 0] [color={rgb, 255:red, 0; green, 0; blue, 0 }  ][fill={rgb, 255:red, 0; green, 0; blue, 0 }  ][line width=0.75]      (0, 0) circle [x radius= 3.35, y radius= 3.35]   ;
\draw    (166.03,91.14) -- (130.48,127.31) ;
\draw [shift={(129.08,128.74)}, rotate = 314.51] [color={rgb, 255:red, 0; green, 0; blue, 0 }  ][line width=0.75]    (10.93,-3.29) .. controls (6.95,-1.4) and (3.31,-0.3) .. (0,0) .. controls (3.31,0.3) and (6.95,1.4) .. (10.93,3.29)   ;
\draw [shift={(167.43,89.72)}, rotate = 134.51] [color={rgb, 255:red, 0; green, 0; blue, 0 }  ][line width=0.75]    (10.93,-3.29) .. controls (6.95,-1.4) and (3.31,-0.3) .. (0,0) .. controls (3.31,0.3) and (6.95,1.4) .. (10.93,3.29)   ;
\draw    (195.21,89.72) -- (195.21,166.1) ;
\draw [shift={(195.21,166.1)}, rotate = 270] [color={rgb, 255:red, 0; green, 0; blue, 0 }  ][line width=0.75]    (0,5.59) -- (0,-5.59)   ;
\draw [shift={(195.21,89.72)}, rotate = 270] [color={rgb, 255:red, 0; green, 0; blue, 0 }  ][line width=0.75]    (0,5.59) -- (0,-5.59)   ;

\draw (117.58,133.38) node [anchor=north west][inner sep=0.75pt]    {$\mathbf{y}_{i}$};
\draw (92.57,92.68) node [anchor=north west][inner sep=0.75pt]    {$h\sin( \alpha )$};
\draw (195.17,106.46) node [anchor=north west][inner sep=0.75pt]    {$\frac{2h\sin( \alpha )}{\sqrt{2}}$};

\end{tikzpicture}

        \caption{The smallest distance apart that the collocation points can be in $\bR^{2}$ such that they are not contained in a ball of radius $h\sin(\alpha)$ on a regular grid. We see that this distance is $2h\sin(\alpha)/\sqrt{2}$.}
        \label{figBall}
    \end{figure}\par
    In the choice of $N$ as given in (\ref{eqnlemConeCondBoundprf2}) there are two cases, one is when the distance $2h\sin(\alpha)/\sqrt{d}$ is greater than one, and the other is when it is between zero and one. We consider both cases and in both show that when $n>{N}$ the inequality (\ref{eqnlemConeCondBoundprf2.1}) holds.\medskip\\
    \textit{Case 1:} $2h\sin(\alpha)/\sqrt{d}$ is greater than one. The collocation points are chosen to be $\Xin=2^{-n}\bZ^{d}\cap{C_{i}}$ with $n\in\bN$. The distance collocation points are from each other is $2^{-n}\leq{1/2}$. Therefore, when $2h\sin(\alpha)/\sqrt{d}$ is greater than one (\ref{eqnlemConeCondBoundprf2.1}) holds.\medskip\\ 
    \textit{Case 2:} $2h\sin(\alpha)/\sqrt{d}$ is between zero and one. As $n>{N}$ we have
    \begin{equation}
        \label{eqnlemConeCondBoundprf3}
        n>\log_{2}\left(\frac{\sqrt{d}}{2h\sin(\alpha)}\right)\Rightarrow\frac{1}{2^{n}}<\frac{2h\sin(\alpha)}{\sqrt{d}}.\nonumber
    \end{equation}
    Hence, the inequality (\ref{eqnlemConeCondBoundprf2.1}) holds.\par
    Therefore, when $n>N$ where $N$ is defined in (\ref{eqnlemConeCondBoundprf2}), there exists a collocation point $\xxiln$ in the ball $B(\yyi,h\sin(\alpha))$.\medskip\\
    \textbf{Step 2:} Now let $\mathbf{x}_{i}\in{C_{i}}$. By Step 1 there exists a collocation point $\xxiln$ within the ball $B(\mathbf{y}_{i},h\sin(\alpha))$. As $\mathbf{y}_{i}:=\mathbf{x}_{i}+h\xi$,
    \begin{equation}
        \label{eqnlemConeCondBoundprf4}
        \|\yyi-\xxi\|_{2}=h\|\xi\|_{2}=h.
    \end{equation}
    Using the radius of the ball we have that
    \begin{equation}
        \label{eqnlemConeCondBoundprf5}
        \|\yyi-\xxiln\|_{2}\leq{h\sin(\alpha)}.
    \end{equation}
    Now we can use (\ref{eqnlemConeCondBoundprf1}), (\ref{eqnlemConeCondBoundprf4}) and (\ref{eqnlemConeCondBoundprf5}), and the triangle inequality to conclude that
    \begin{equation}
        \label{eqnlemConeCondBoundprf6}
        \|\xxi-\xxiln\|_{2}\leq\|\yyi-\xxi\|_{2}+\|\yyi-\xxiln\|_{2}\leq{h(1+\sin(\alpha))}\leq{B}.\nonumber
    \end{equation}
    This proves the lemma.
    \end{proof}

Now we can use this lemma to see that choosing the collocation points as $Z_{n}=2^{-n}\bZ^{d}\cap\Omega$ in Theorem \ref{thmConvergence} is a sufficient condition to satisfy the assumption on the fill distance.
\begin{cor}
    \label{corFillDistance}
    Choose $Z_{n}=2^{-n}\bZ^{d}\cap\Omega$ with $\Xin:=Z_{n}\cap{C_{i}}$, where $\xxiln$ are elements of $\Xin$ for $\ell=1,\dots,\Lin$. Assume the sets $C_{i}$ satisfy an interior cone condition. Then the fill distance $h_{\Xin,C_{i}}$ defined in (\ref{eqnthmConvergence3}) satisfies
    \begin{equation}
        \label{eqncorFillDistance1}
        \lim_{n\rightarrow\infty}h_{\Xin,C_{i}}=0.\nonumber
    \end{equation}
\end{cor}
\begin{proof}
    The result of Lemma \ref{lemConeCondBound} implies that for any $B>0$ there exists an $N\in\bN_{0}$ such that for all $n>N$ and $\xxi\in{C_{i}}$ there exists a collocation point $\xxiln\in\Xin$ such that
    \begin{equation}
    \label{eqnSupMinDiffBdd}
        h_{\Xin,C_{i}}=\sup_{\xxi\in{C_{i}}}\min_{\ell=1,\dots,\Lin}\|\xxi-\xxiln\|_{2}\leq{B}.\nonumber
    \end{equation}
\end{proof}
This means that when the set $C_{i}$ satisfies an interior cone condition, which is guaranteed if $C_{i}$ has Lipschitz boundary by Section 4.7 on page 67 of \cite{adams2003sobolev}, then if the collocation points are chosen as $\Xin:=Z_{n}\cap{C_{i}}$ the fill distance defined in (\ref{eqnthmConvergence3}) tends to zero as $n$ tends to infinity.

\section{Construction using quadratic programming}
\label{secConstruction}

In this section we show that the minimisation problem (\ref{eqnMinProbi}) defined in Section \ref{secDiscret} is equivalent to the minimisation problem (\ref{eqnMinProb}), which appears as equation (8) in \cite{ward_construction_nodate}. Hence, we can apply Theorem \ref{thmConvergence} and see that the minimisation problem (\ref{eqnMinProb}) converges as the fill distance decreases. We state techniques employed previously in \cite{ward_construction_nodate} to find a solution to (\ref{eqnMinProb}). This solution is a Lyapunov function for the system (\ref{eqnfis}).\par
The techniques to find a solution to (\ref{eqnMinProbi}) are summarised in the following algorithm that can also be found in \cite{ward_construction_nodate}; this algorithm is implemented using MATLAB code for the examples in Section \ref{secExamples}.

\begin{algorithm}[H]
\caption{Lyapunov functions for switched systems}
\label{algMain}
    \begin{algorithmic}[1]
        \REQUIRE State space $\Omega$ and partition $\Ptt$ for $\theta=1,\dots,\Theta$ (see Section \ref{secPartition}), functions $\mathbf{f}_{i}(\mathbf{x})$ (\ref{eqnfis}), Wendland function $\Psi_0(r)$ (see Section \ref{secRKHS}), vector $\mathbf{b}$ (\ref{eqnbTAb}) and evaluation points.
        \STATE Fix collocation points $Y^{\theta}=\{\ymtt\}_{m=1}^{\Mtt}$ in each region $\OPtt$, $\theta=1,\dots,\Theta$.
        \STATE If $\mathbf{0}$ was selected as a collocation point remove it.
        \STATE For each collocation point find a basis of the vectors $\mathbf{f}_{\pjtt}(\ymtt)$ and use this to find the values of the coefficients $\aassjmtt$ (\ref{eqnllssmtt}).
        \STATE Find $Q$ (\ref{eqnQ}).
        \STATE Find $A$ (\ref{eqnAElem}).
        \STATE Solve a quadratic programming problem for $\boldsymbol{\beta}$ (\ref{eqnbTAb}).
        \STATE Find values of $v$ (\ref{eqnLyapv}) at evaluation points and plot. For all $i\in{I}$ find values of the orbital derivatives $\nabla{v}\cdot\mathbf{f}_{i}$ at evaluation points in $C_{i}$ and plot.
     \end{algorithmic}
\end{algorithm}
Section \ref{secMinProb} provides details on the collocation points $Y^{\theta}$ and how these relate to the collocation points $X_{i}$ given in Section \ref{secDiscret}. This section also contains Lemma \ref{lemEquivMinProbs}, where we show that the minimisation problems (\ref{eqnMinProbi}) and (\ref{eqnMinProb}) are equivalent. The methods contained within \cite{ward_construction_nodate} to find a solution to the minimisation problem (\ref{eqnMinProb}) are explained briefly in Sections \ref{secBasis} and \ref{secQuadProg}. With regard to Algorithm \ref{algMain}, details on steps $3$ and $4$ can be found in Section \ref{secBasis}, and on steps $5$, $6$ and $7$ in Section \ref{secQuadProg}.

\subsection{Partitioning the state space}
\label{secPartition}

We define the following sets from \cite{ward_construction_nodate} for ${P}\subseteq{I}$, $P\neq\emptyset$
\begin{equation}
    \label{eqnOP}
    \OP:=\left(\bigcap_{i\in{P}}C_{i}\right)\cap\left(\bigcap_{j\in{I\setminus{P}}}C_{j}^{c}\right).
\end{equation}
Figure \ref{fig:PartitionExampleOmegas} illustrates a partition for the switched system (\ref{eqnExampleSwitchedTrajectory}), visualised previously in Figure \ref{fig:PartitionExampleCis}. Here $I=\{1,2\}$ and the sets corresponding to $P^{1}=\{1\}$, $P^{2}=\{1,2\}$ and $P^{3}=\{2\}$ are nonempty and partition the state space. These are defined as $\Omega^{\{1\}}=C_{1}\cap{C_{2}^{c}}$, $\Omega^{\{1,2\}}=C_{1}\cap{C_{2}}$ and $\Omega^{\{2\}}={C_{1}^{c}}\cap{C_{2}}$.\par
We now give a lemma that did not appear in \cite{ward_construction_nodate} showing that the sets $\OP$ (\ref{eqnOP}) are a partition of $\Omega$.
\begin{lemma}
    \label{lemPartOP}
    The sets $\OP$ defined in (\ref{eqnOP}) for all $\emptyset\neq{P}\subseteq{I}$ are a partition of $\Omega$, i.e.
    \begin{equation}
    \label{eqnlemPartOP1}
    \Omega=\bigcup_{\emptyset\neq{P}\subseteq{I}}\OP
    \end{equation}
    where $\Omega^{P_{1}}\cap\Omega^{P_{2}}=\emptyset$ for all subsets $P_{1},P_{2}\subseteq{I}$ with $P_{1}\neq{P_{2}}$.
\end{lemma}
\begin{proof}
    First we show (\ref{eqnlemPartOP1}), which we do by showing that
    \begin{equation}
    \label{eqnlemprfPartOP1}
        \bigcup_{\emptyset\neq{P}\subseteq{I}}\OP\subseteq\Omega\text{ and }\bigcup_{\emptyset\neq{P}\subseteq{I}}\OP\supseteq\Omega.
    \end{equation}
    The first statement is immediate due to the definition of $\OP$ in (\ref{eqnOP}).
    To show the second statement we take $\mathbf{y}\in\Omega$. Due to (\ref{eqnCoveringCi}), we know that $\mathbf{y}$ must be contained in at least one of the sets $C_{i}$ for $i\in{I}$, and we denote $P$ as the nonempty subset of $I$ that contains exactly all of the $i$'s such that $\mathbf{y}$ is in $C_{i}$. This means that for $j\in{I}\setminus{P}$, $\mathbf{y}\notin{C_{j}}$. Therefore, $\mathbf{y}\in\OP$ by the definition in (\ref{eqnOP}), so we have that
    \begin{equation}
        \label{eqnlemprfPartOP2}
        \mathbf{y}\in\bigcup_{\emptyset\neq{P}\subseteq{I}}\OP
    \end{equation}
    showing the second statement from (\ref{eqnlemprfPartOP1}). This tells us that (\ref{eqnlemPartOP1}) holds.\par
    To prove the second part of the lemma we assume without loss of generality that there exists a $\mathbf{y}$ that is in both $\Omega^{P_{1}}$ and $\Omega^{P_{2}}$ for $P_{1}\neq{P_{2}}$. As $P_{1}$ and $P_{2}$ are distinct, there must be an element $i$ that is in exactly one of $P_{1}$ or $P_{2}$. Consequently, $\mathbf{y}\in{C_{i}}$ and $\mathbf{y}\in{C_{i}^{c}}$; however, this is not possible, so there is no $\mathbf{y}$ that is in both $\Omega^{P_{1}}$ and $\Omega^{P_{2}}$ which implies $\Omega^{P_{1}}\cap\Omega^{P_{2}}=\emptyset$.
\end{proof}

For our algorithm we refer to the total number of subsets $P$ that induce a nonempty $\OP$ as $\Theta$, i.e. in the case of Figure \ref{fig:PartitionExampleOmegas} we would have $\Theta=3$. This number is bounded above by $2^{k}-1$ as there are $k$ elements in the set $I$ and we cannot have $P=\emptyset$. This enables us to label these subsets as $\Ptt$ where $\theta=1,\dots,\Theta$ and we can write the elements of each of these sets as $\pjtt$, with $j=1,\dots,\cPtt$ where $\cPtt:=|\Ptt|$. In Figure \ref{fig:PartitionExampleOmegas} we have labelled $P^{1}=\{1\}$ with $p_{1}^{1}=1$, $P^{2}=\{1,2\}$ with $p_{1}^{2}=1$ and $p_{2}^{2}=2$ and $P^{3}=\{2\}$ with $p_{1}^{3}=2$. 

\begin{figure}[h]
    \centering

\tikzset{every picture/.style={line width=0.75pt}} 

\begin{tikzpicture}[x=0.75pt,y=0.75pt,yscale=-1,xscale=1]

\draw  [fill={rgb, 255:red, 74; green, 144; blue, 226 }  ,fill opacity=0.33 ] (16.36,28.9) .. controls (16.9,29.17) and (235.11,28.32) .. (236.58,29.79) .. controls (238.05,31.26) and (114.59,56.25) .. (120.47,88.58) .. controls (126.35,120.91) and (248.33,129.73) .. (240.99,157.66) .. controls (233.64,185.58) and (202.59,207) .. (201.46,207) .. controls (200.32,207) and (15.02,205.86) .. (16.36,206.87) .. controls (17.7,207.88) and (15.82,28.64) .. (16.36,28.9) -- cycle ;
\draw  [fill={rgb, 255:red, 208; green, 2; blue, 27 }  ,fill opacity=0.32 ] (16.36,29.9) .. controls (17.59,30.79) and (150.38,30.45) .. (151.45,30.45) .. controls (152.53,30.45) and (284.12,29.91) .. (283.31,29.9) .. controls (282.5,29.9) and (284.12,208.42) .. (283.31,207.87) .. controls (282.5,207.33) and (173.38,207.16) .. (171.91,207.16) .. controls (170.44,207.16) and (83.73,163.07) .. (94.02,121.91) .. controls (104.3,80.76) and (16.12,71.94) .. (16.12,71.94) .. controls (16.12,71.94) and (15.13,29.02) .. (16.36,29.9) -- cycle ;
\draw  [fill={rgb, 255:red, 74; green, 144; blue, 226 }  ,fill opacity=0.32 ][line width=0.75]  (301.24,42.82) -- (330.64,42.82) -- (330.64,72.21) -- (301.24,72.21) -- cycle ;
\draw  [fill={rgb, 255:red, 208; green, 2; blue, 27 }  ,fill opacity=0.3 ][line width=0.75]  (301.24,101.6) -- (330.64,101.6) -- (330.64,131) -- (301.24,131) -- cycle ;
\draw  [line width=1.5]  (16.36,28.82) -- (283.31,28.82) -- (283.31,206.79) -- (16.36,206.79) -- cycle ;
\draw  [fill={rgb, 255:red, 74; green, 144; blue, 226 }  ,fill opacity=0.32 ][line width=0.75]  (301.24,101.6) -- (330.64,101.6) -- (330.64,131) -- (301.24,131) -- cycle ;
\draw  [fill={rgb, 255:red, 208; green, 2; blue, 27 }  ,fill opacity=0.3 ][line width=0.75]  (301.24,160.39) -- (330.64,160.39) -- (330.64,189.79) -- (301.24,189.79) -- cycle ;

\draw (332.64,105) node [anchor=north west][inner sep=0.75pt]  [font=\normalsize]  {$=C_{1} \cap C_{2} =\Omega ^{\{1,2\}} =\Omega {^{P}}^{2}$};
\draw (263.72,13.06) node [anchor=north west][inner sep=0.75pt]    {$\Omega $};
\draw (332.64,46.22) node [anchor=north west][inner sep=0.75pt]  [font=\normalsize]  {$=C_{1} \cap C_{2}^{c} =\Omega ^{\{1\}} =\Omega {^{P}}^{1}$};
\draw (332.64,163.79) node [anchor=north west][inner sep=0.75pt]  [font=\normalsize]  {$=C_{1}^{c} \cap C_{2} =\Omega ^{\{2\}} =\Omega {^{P}}^{3}$};

\end{tikzpicture}

    \caption{The example (\ref{eqnExampleSwitchedTrajectory}) from Figure \ref{fig:PartitionExampleCis} showing how the subsets $\OP$ are defined using (\ref{eqnOP}). In this case we have $\Theta=3$ and we label $P^{1}=\{1\}$, $P^{2}=\{1,2\}$ and $P^{3}=\{2\}$.}
    \label{fig:PartitionExampleOmegas}
\end{figure}
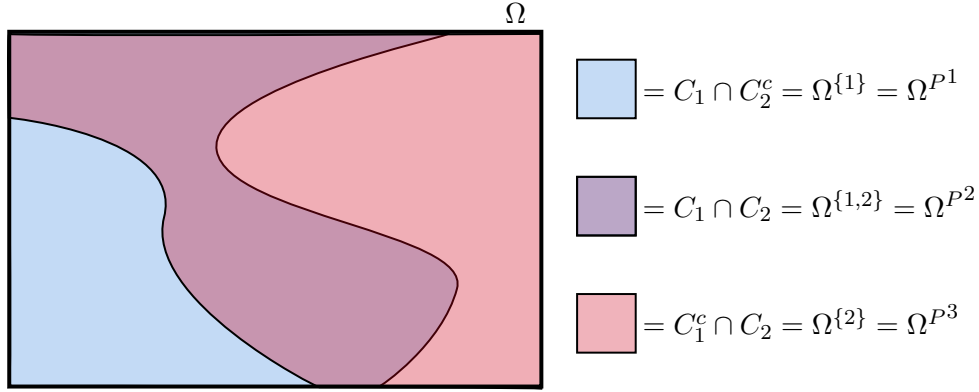

\subsection{Equivalent minimisation problem}
\label{secMinProb}

Using this set-up we now show that the minimisation problem (\ref{eqnMinProbi}) is equivalent to the minimisation problem (8) from \cite{ward_construction_nodate}. The minimisation problem (\ref{eqnMinProbi}) has constraints that need to be satisfied for each collocation point $\xxil\in{X_{i}}$ for all $i=1,\dots,k$ and $\ell=1,\dots,L_{i}$. On the other hand, in the minimisation problem from \cite{ward_construction_nodate} the constraints are imposed at each collocation point in each set $\OPtt$ for $\theta=1,\dots,\Theta$. Therefore, we need to fix collocation points in the sets $\OPtt$ and show how these are related to the collocation points $X_{i}$ in the sets $C_{i}$. We define the collocation points in the set $\OPtt$ to be
\begin{equation}
    \label{eqnCollPts}
    Y^{\theta}:=Z\cap\OPtt=\{\mathbf{y}^{1,\theta},\dots,\mathbf{y}^{\Mtt,\theta}\}
\end{equation}
where $Z$ is the set of collocation points in $\Omega$ (\ref{eqnCollPtsOmega}). Here $\Mtt$ is the number of collocation points in $\OPtt$.\par
In Figure \ref{fig:PartitionGridpoints} we show how the collocation points $Z$ (\ref{eqnCollPtsOmega}), $X_{i}$ (\ref{eqnCollPtsi}) for $i=1,\dots,k$ and $Y^{\theta}$ for $\theta=1,\dots,\Theta$ are related using the example discussed before in Figures \ref{fig:PartitionExampleCis} and \ref{fig:PartitionExampleOmegas}.
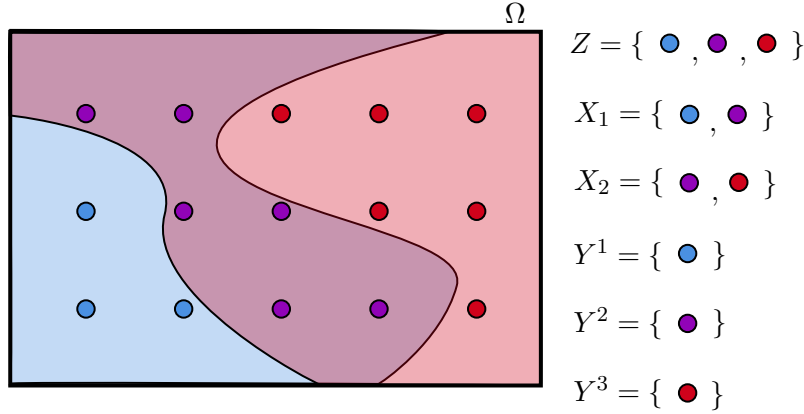
\begin{figure}[h]
    \centering

\tikzset{every picture/.style={line width=0.75pt}} 

\begin{tikzpicture}[x=0.75pt,y=0.75pt,yscale=-1,xscale=1]

\draw  [fill={rgb, 255:red, 74; green, 144; blue, 226 }  ,fill opacity=0.33 ] (40.86,58.22) .. controls (41.4,58.48) and (258.99,57.64) .. (260.46,59.1) .. controls (261.93,60.57) and (138.81,85.48) .. (144.68,117.73) .. controls (150.54,149.97) and (272.18,158.76) .. (264.86,186.61) .. controls (257.53,214.46) and (226.57,235.82) .. (225.44,235.82) .. controls (224.31,235.82) and (39.52,234.68) .. (40.86,235.69) .. controls (42.2,236.7) and (40.32,57.96) .. (40.86,58.22) -- cycle ;
\draw  [fill={rgb, 255:red, 208; green, 2; blue, 27 }  ,fill opacity=0.32 ] (40.86,59.22) .. controls (42.09,60.1) and (174.5,59.76) .. (175.57,59.76) .. controls (176.65,59.76) and (307.87,59.23) .. (307.06,59.22) .. controls (306.26,59.21) and (307.87,237.24) .. (307.06,236.69) .. controls (306.26,236.14) and (197.44,235.98) .. (195.97,235.98) .. controls (194.51,235.98) and (108.04,192.01) .. (118.3,150.97) .. controls (128.56,109.93) and (40.62,101.14) .. (40.62,101.14) .. controls (40.62,101.14) and (39.63,58.33) .. (40.86,59.22) -- cycle ;
\draw  [line width=1.5]  (40.86,59.14) -- (307.06,59.14) -- (307.06,236.61) -- (40.86,236.61) -- cycle ;
\draw  [color={rgb, 255:red, 0; green, 0; blue, 0 }  ,draw opacity=1 ][fill={rgb, 255:red, 74; green, 144; blue, 226 }  ,fill opacity=1 ] (123.5,198.22) .. controls (123.5,195.81) and (125.45,193.86) .. (127.86,193.86) .. controls (130.27,193.86) and (132.22,195.81) .. (132.22,198.22) .. controls (132.22,200.63) and (130.27,202.58) .. (127.86,202.58) .. controls (125.45,202.58) and (123.5,200.63) .. (123.5,198.22) -- cycle ;
\draw  [color={rgb, 255:red, 0; green, 0; blue, 0 }  ,draw opacity=1 ][fill={rgb, 255:red, 74; green, 144; blue, 226 }  ,fill opacity=1 ] (74.5,149.22) .. controls (74.5,146.81) and (76.45,144.86) .. (78.86,144.86) .. controls (81.27,144.86) and (83.22,146.81) .. (83.22,149.22) .. controls (83.22,151.63) and (81.27,153.58) .. (78.86,153.58) .. controls (76.45,153.58) and (74.5,151.63) .. (74.5,149.22) -- cycle ;
\draw  [color={rgb, 255:red, 0; green, 0; blue, 0 }  ,draw opacity=1 ][fill={rgb, 255:red, 74; green, 144; blue, 226 }  ,fill opacity=1 ] (74.5,198.22) .. controls (74.5,195.81) and (76.45,193.86) .. (78.86,193.86) .. controls (81.27,193.86) and (83.22,195.81) .. (83.22,198.22) .. controls (83.22,200.63) and (81.27,202.58) .. (78.86,202.58) .. controls (76.45,202.58) and (74.5,200.63) .. (74.5,198.22) -- cycle ;
\draw  [color={rgb, 255:red, 0; green, 0; blue, 0 }  ,draw opacity=1 ][fill={rgb, 255:red, 145; green, 3; blue, 183 }  ,fill opacity=1 ] (74.5,100.22) .. controls (74.5,97.81) and (76.45,95.86) .. (78.86,95.86) .. controls (81.27,95.86) and (83.22,97.81) .. (83.22,100.22) .. controls (83.22,102.63) and (81.27,104.58) .. (78.86,104.58) .. controls (76.45,104.58) and (74.5,102.63) .. (74.5,100.22) -- cycle ;
\draw  [color={rgb, 255:red, 0; green, 0; blue, 0 }  ,draw opacity=1 ][fill={rgb, 255:red, 145; green, 3; blue, 183 }  ,fill opacity=1 ] (123.5,149.22) .. controls (123.5,146.81) and (125.45,144.86) .. (127.86,144.86) .. controls (130.27,144.86) and (132.22,146.81) .. (132.22,149.22) .. controls (132.22,151.63) and (130.27,153.58) .. (127.86,153.58) .. controls (125.45,153.58) and (123.5,151.63) .. (123.5,149.22) -- cycle ;
\draw  [color={rgb, 255:red, 0; green, 0; blue, 0 }  ,draw opacity=1 ][fill={rgb, 255:red, 145; green, 3; blue, 183 }  ,fill opacity=1 ] (123.5,100.22) .. controls (123.5,97.81) and (125.45,95.86) .. (127.86,95.86) .. controls (130.27,95.86) and (132.22,97.81) .. (132.22,100.22) .. controls (132.22,102.63) and (130.27,104.58) .. (127.86,104.58) .. controls (125.45,104.58) and (123.5,102.63) .. (123.5,100.22) -- cycle ;
\draw  [color={rgb, 255:red, 0; green, 0; blue, 0 }  ,draw opacity=1 ][fill={rgb, 255:red, 145; green, 3; blue, 183 }  ,fill opacity=1 ] (172.5,198.22) .. controls (172.5,195.81) and (174.45,193.86) .. (176.86,193.86) .. controls (179.27,193.86) and (181.22,195.81) .. (181.22,198.22) .. controls (181.22,200.63) and (179.27,202.58) .. (176.86,202.58) .. controls (174.45,202.58) and (172.5,200.63) .. (172.5,198.22) -- cycle ;
\draw  [color={rgb, 255:red, 0; green, 0; blue, 0 }  ,draw opacity=1 ][fill={rgb, 255:red, 145; green, 3; blue, 183 }  ,fill opacity=1 ] (172.5,149.22) .. controls (172.5,146.81) and (174.45,144.86) .. (176.86,144.86) .. controls (179.27,144.86) and (181.22,146.81) .. (181.22,149.22) .. controls (181.22,151.63) and (179.27,153.58) .. (176.86,153.58) .. controls (174.45,153.58) and (172.5,151.63) .. (172.5,149.22) -- cycle ;
\draw  [color={rgb, 255:red, 0; green, 0; blue, 0 }  ,draw opacity=1 ][fill={rgb, 255:red, 145; green, 3; blue, 183 }  ,fill opacity=1 ] (221.5,198.22) .. controls (221.5,195.81) and (223.45,193.86) .. (225.86,193.86) .. controls (228.27,193.86) and (230.22,195.81) .. (230.22,198.22) .. controls (230.22,200.63) and (228.27,202.58) .. (225.86,202.58) .. controls (223.45,202.58) and (221.5,200.63) .. (221.5,198.22) -- cycle ;
\draw  [color={rgb, 255:red, 0; green, 0; blue, 0 }  ,draw opacity=1 ][fill={rgb, 255:red, 208; green, 2; blue, 27 }  ,fill opacity=1 ] (270.5,149.22) .. controls (270.5,146.81) and (272.45,144.86) .. (274.86,144.86) .. controls (277.27,144.86) and (279.22,146.81) .. (279.22,149.22) .. controls (279.22,151.63) and (277.27,153.58) .. (274.86,153.58) .. controls (272.45,153.58) and (270.5,151.63) .. (270.5,149.22) -- cycle ;
\draw  [color={rgb, 255:red, 0; green, 0; blue, 0 }  ,draw opacity=1 ][fill={rgb, 255:red, 208; green, 2; blue, 27 }  ,fill opacity=1 ] (270.5,100.22) .. controls (270.5,97.81) and (272.45,95.86) .. (274.86,95.86) .. controls (277.27,95.86) and (279.22,97.81) .. (279.22,100.22) .. controls (279.22,102.63) and (277.27,104.58) .. (274.86,104.58) .. controls (272.45,104.58) and (270.5,102.63) .. (270.5,100.22) -- cycle ;
\draw  [color={rgb, 255:red, 0; green, 0; blue, 0 }  ,draw opacity=1 ][fill={rgb, 255:red, 208; green, 2; blue, 27 }  ,fill opacity=1 ] (221.5,100.22) .. controls (221.5,97.81) and (223.45,95.86) .. (225.86,95.86) .. controls (228.27,95.86) and (230.22,97.81) .. (230.22,100.22) .. controls (230.22,102.63) and (228.27,104.58) .. (225.86,104.58) .. controls (223.45,104.58) and (221.5,102.63) .. (221.5,100.22) -- cycle ;
\draw  [color={rgb, 255:red, 0; green, 0; blue, 0 }  ,draw opacity=1 ][fill={rgb, 255:red, 208; green, 2; blue, 27 }  ,fill opacity=1 ] (221.5,149.22) .. controls (221.5,146.81) and (223.45,144.86) .. (225.86,144.86) .. controls (228.27,144.86) and (230.22,146.81) .. (230.22,149.22) .. controls (230.22,151.63) and (228.27,153.58) .. (225.86,153.58) .. controls (223.45,153.58) and (221.5,151.63) .. (221.5,149.22) -- cycle ;
\draw  [color={rgb, 255:red, 0; green, 0; blue, 0 }  ,draw opacity=1 ][fill={rgb, 255:red, 208; green, 2; blue, 27 }  ,fill opacity=1 ] (270.5,198.22) .. controls (270.5,195.81) and (272.45,193.86) .. (274.86,193.86) .. controls (277.27,193.86) and (279.22,195.81) .. (279.22,198.22) .. controls (279.22,200.63) and (277.27,202.58) .. (274.86,202.58) .. controls (272.45,202.58) and (270.5,200.63) .. (270.5,198.22) -- cycle ;
\draw  [color={rgb, 255:red, 0; green, 0; blue, 0 }  ,draw opacity=1 ][fill={rgb, 255:red, 208; green, 2; blue, 27 }  ,fill opacity=1 ] (172.5,100.22) .. controls (172.5,97.81) and (174.45,95.86) .. (176.86,95.86) .. controls (179.27,95.86) and (181.22,97.81) .. (181.22,100.22) .. controls (181.22,102.63) and (179.27,104.58) .. (176.86,104.58) .. controls (174.45,104.58) and (172.5,102.63) .. (172.5,100.22) -- cycle ;
\draw  [color={rgb, 255:red, 0; green, 0; blue, 0 }  ,draw opacity=1 ][fill={rgb, 255:red, 74; green, 144; blue, 226 }  ,fill opacity=1 ] (367.33,64.92) .. controls (367.33,62.51) and (369.28,60.56) .. (371.69,60.56) .. controls (374.1,60.56) and (376.05,62.51) .. (376.05,64.92) .. controls (376.05,67.33) and (374.1,69.28) .. (371.69,69.28) .. controls (369.28,69.28) and (367.33,67.33) .. (367.33,64.92) -- cycle ;
\draw  [color={rgb, 255:red, 0; green, 0; blue, 0 }  ,draw opacity=1 ][fill={rgb, 255:red, 145; green, 3; blue, 183 }  ,fill opacity=1 ] (391.17,64.92) .. controls (391.17,62.51) and (393.12,60.56) .. (395.53,60.56) .. controls (397.94,60.56) and (399.89,62.51) .. (399.89,64.92) .. controls (399.89,67.33) and (397.94,69.28) .. (395.53,69.28) .. controls (393.12,69.28) and (391.17,67.33) .. (391.17,64.92) -- cycle ;
\draw  [color={rgb, 255:red, 0; green, 0; blue, 0 }  ,draw opacity=1 ][fill={rgb, 255:red, 208; green, 2; blue, 27 }  ,fill opacity=1 ] (416,64.92) .. controls (416,62.51) and (417.96,60.56) .. (420.36,60.56) .. controls (422.77,60.56) and (424.72,62.51) .. (424.72,64.92) .. controls (424.72,67.33) and (422.77,69.28) .. (420.36,69.28) .. controls (417.96,69.28) and (416,67.33) .. (416,64.92) -- cycle ;
\draw  [color={rgb, 255:red, 0; green, 0; blue, 0 }  ,draw opacity=1 ][fill={rgb, 255:red, 74; green, 144; blue, 226 }  ,fill opacity=1 ] (377.33,100.63) .. controls (377.33,98.22) and (379.28,96.27) .. (381.69,96.27) .. controls (384.1,96.27) and (386.05,98.22) .. (386.05,100.63) .. controls (386.05,103.04) and (384.1,104.99) .. (381.69,104.99) .. controls (379.28,104.99) and (377.33,103.04) .. (377.33,100.63) -- cycle ;
\draw  [color={rgb, 255:red, 0; green, 0; blue, 0 }  ,draw opacity=1 ][fill={rgb, 255:red, 145; green, 3; blue, 183 }  ,fill opacity=1 ] (401.17,100.63) .. controls (401.17,98.22) and (403.12,96.27) .. (405.53,96.27) .. controls (407.94,96.27) and (409.89,98.22) .. (409.89,100.63) .. controls (409.89,103.04) and (407.94,104.99) .. (405.53,104.99) .. controls (403.12,104.99) and (401.17,103.04) .. (401.17,100.63) -- cycle ;
\draw  [color={rgb, 255:red, 0; green, 0; blue, 0 }  ,draw opacity=1 ][fill={rgb, 255:red, 145; green, 3; blue, 183 }  ,fill opacity=1 ] (377.17,134.92) .. controls (377.17,132.51) and (379.12,130.56) .. (381.53,130.56) .. controls (383.94,130.56) and (385.89,132.51) .. (385.89,134.92) .. controls (385.89,137.33) and (383.94,139.28) .. (381.53,139.28) .. controls (379.12,139.28) and (377.17,137.33) .. (377.17,134.92) -- cycle ;
\draw  [color={rgb, 255:red, 0; green, 0; blue, 0 }  ,draw opacity=1 ][fill={rgb, 255:red, 208; green, 2; blue, 27 }  ,fill opacity=1 ] (402,134.69) .. controls (402,132.28) and (403.96,130.33) .. (406.36,130.33) .. controls (408.77,130.33) and (410.72,132.28) .. (410.72,134.69) .. controls (410.72,137.1) and (408.77,139.05) .. (406.36,139.05) .. controls (403.96,139.05) and (402,137.1) .. (402,134.69) -- cycle ;
\draw  [color={rgb, 255:red, 0; green, 0; blue, 0 }  ,draw opacity=1 ][fill={rgb, 255:red, 74; green, 144; blue, 226 }  ,fill opacity=1 ] (376.33,170.4) .. controls (376.33,167.99) and (378.28,166.04) .. (380.69,166.04) .. controls (383.1,166.04) and (385.05,167.99) .. (385.05,170.4) .. controls (385.05,172.81) and (383.1,174.76) .. (380.69,174.76) .. controls (378.28,174.76) and (376.33,172.81) .. (376.33,170.4) -- cycle ;
\draw  [color={rgb, 255:red, 0; green, 0; blue, 0 }  ,draw opacity=1 ][fill={rgb, 255:red, 145; green, 3; blue, 183 }  ,fill opacity=1 ] (376.33,205.46) .. controls (376.33,203.05) and (378.28,201.1) .. (380.69,201.1) .. controls (383.1,201.1) and (385.05,203.05) .. (385.05,205.46) .. controls (385.05,207.86) and (383.1,209.82) .. (380.69,209.82) .. controls (378.28,209.82) and (376.33,207.86) .. (376.33,205.46) -- cycle ;
\draw  [color={rgb, 255:red, 0; green, 0; blue, 0 }  ,draw opacity=1 ][fill={rgb, 255:red, 208; green, 2; blue, 27 }  ,fill opacity=1 ] (376.33,240.34) .. controls (376.33,237.93) and (378.28,235.98) .. (380.69,235.98) .. controls (383.1,235.98) and (385.05,237.93) .. (385.05,240.34) .. controls (385.05,242.75) and (383.1,244.7) .. (380.69,244.7) .. controls (378.28,244.7) and (376.33,242.75) .. (376.33,240.34) -- cycle ;

\draw (287.5,43.4) node [anchor=north west][inner sep=0.75pt]    {$\Omega $};
\draw (319.97,57.32) node [anchor=north west][inner sep=0.75pt]  [font=\normalsize]  {$Z=\{$};
\draw (322.28,92.03) node [anchor=north west][inner sep=0.75pt]  [font=\normalsize]  {$X_{1} =\{$};
\draw (322.28,126.09) node [anchor=north west][inner sep=0.75pt]  [font=\normalsize]  {$X_{2} =\{$};
\draw (321.95,161.8) node [anchor=north west][inner sep=0.75pt]  [font=\normalsize]  {$Y^{1} =\{$};
\draw (321.95,195.86) node [anchor=north west][inner sep=0.75pt]  [font=\normalsize]  {$Y^{2} =\{$};
\draw (321.95,230.74) node [anchor=north west][inner sep=0.75pt]  [font=\normalsize]  {$Y^{3} =\{$};
\draw (381,70.32) node [anchor=north west][inner sep=0.75pt]    {$,$};
\draw (405,70.32) node [anchor=north west][inner sep=0.75pt]    {$,$};
\draw (431,57.32) node [anchor=north west][inner sep=0.75pt]    {$\}$};
\draw (390,106.03) node [anchor=north west][inner sep=0.75pt]    {$,$};
\draw (416,93.03) node [anchor=north west][inner sep=0.75pt]    {$\}$};
\draw (391,140.32) node [anchor=north west][inner sep=0.75pt]    {$,$};
\draw (417,127.09) node [anchor=north west][inner sep=0.75pt]    {$\}$};
\draw (391.53,162.8) node [anchor=north west][inner sep=0.75pt]    {$\}$};
\draw (391.53,196.86) node [anchor=north west][inner sep=0.75pt]    {$\}$};
\draw (390.53,232.74) node [anchor=north west][inner sep=0.75pt]    {$\}$};

\end{tikzpicture}

    \caption{Showing how the sets of collocation points would be defined in the example from Figures \ref{fig:PartitionExampleCis} and \ref{fig:PartitionExampleOmegas}. The whole set of collocation points is the set $Z$ which comprises the blue, purple and red points. The collocation points $X_{1}$ for the set $C_{1}$ are the blue and purple points, and the points $X_{2}$ for the set $C_{2}$ are the purple and red points. The subsets $\Omega^{P^{1}}$, $\Omega^{P^{2}}$ and $\Omega^{P^{3}}$ are defined in Figure \ref{fig:PartitionExampleOmegas}, and under this definition $\Omega^{P^{1}}$ is the blue set and the collocation points $Y^{1}$ are the blue points, $\Omega^{P^{2}}$ is the purple set and the collocation points $Y^{2}$ are the purple points and $\Omega^{P^{3}}$ is the red set and the collocation points $Y^{3}$ are the red points.}
    \label{fig:PartitionGridpoints}
\end{figure}
We can now express some properties of the collocation points $Z$, $X_{i}$ and $Y^{\theta}$. In the same way as the sets $X_{i}$, the union of the sets $Y^{\theta}$ for $\theta=1,\dots,\Theta$ is the whole set of collocation points $Z$. This is written as
\begin{equation}
    \label{eqnZCollPtsUnionYCollPts}
    Z=\bigcup_{\theta=1,\dots,\Theta}Y^{\theta}.\nonumber
\end{equation}
We can also state that if $i\in{I}$ is an element of the sets $P^{\eta_{r}}\subseteq{I}$ for $1\leq\eta_{r}\leq\Theta$, $r=1,\dots,R$ and no other sets $\Ptt$ for $1\leq\theta\leq\Theta$, i.e.
\begin{equation}
    \label{eqnCiUnionOmegaPeen}
    C_{i}=\bigcup_{r=1,\dots,R}\Omega^{P^{\eta_{r}}},\nonumber
\end{equation}
we have that
\begin{equation}
    \label{eqnXCollPtsUnionYCollPts}
    X_{i}=\bigcup_{r=1,\dots,R}Y^{\eta_{r}}\phantom{vv}\text{ and }\phantom{vv}\Li=\sum_{r=1,\dots,R}M^{\eta_{r}}.
\end{equation}
We can see this in the example shown in Figure \ref{fig:PartitionGridpoints}. The set $C_{1}$ is the union of the sets $\Omega^{P^{1}}$ and $\Omega^{P^{2}}$, and here the set of collocation points $X_{1}$ for the set $C_{1}$ is the union of the collocation points $Y^{1}$ and $Y^{2}$, which are the collocation points for the sets $\Omega^{P^{1}}$ and $\Omega^{P^{2}}$, respectively. Also, the number of collocation points $L^{1}$ for the set $C_{1}$ is equal to the sum of the number of collocation points for the sets $\Omega^{P^{1}}$ and $\Omega^{P^{2}}$, i.e. $L^{1}=M^{1}+M^{2}$.
Finally, the sets of collocation points $Y^{\theta}$ are distinct as the sets $\OPtt$ are distinct, which means
\begin{equation}
    \label{eqnZCollSumYColl}
    N=\sum_{\theta=1,\dots,\Theta}M^{\theta}\nonumber
\end{equation}
where $N$ is the total number of collocation points (\ref{eqnCollPtsOmega}).\par
These collocation points have been taken in the same way as in \cite{ward_construction_nodate}, which means we can use the techniques from this article to find a minimisation problem. This is the same method as used in Section \ref{secDiscret}. Define for all $\theta=1\dots,\Theta$ and $j=1,\dots,\cPtt$ the linear differential operator
\begin{equation}
    \label{eqnDpjtt}
    D_{\pjtt}V(\mathbf{y}):=\nabla_{\mathbf{y}}
    V(\mathbf{y})\cdot\mathbf{f}_{\pjtt}(\mathbf{y})\text{ }\forall\mathbf{y}\in\OPtt
\end{equation}
where for each $i\in I$ we assume that $\mathbf{f}_{\pjtt}$ is the restriction to $C_{\pjtt}$ of a $C^{\alpha-1}$ vector field defined on an open neighbourhood of $C_{\pjtt}$ for $j=1,\dots,\cPtt$ and $\alpha>d/2+2$. The system of partial differential inequalities we wish to solve is
\begin{equation}
    \label{eqnDV}
    D_{\pjtt}V(\mathbf{y})\leq{b_{\pjtt}}(\mathbf{y})
\end{equation}
where $b_{\pjtt}(\mathbf{y})$ are continuous functions that are negative for all $\mathbf{y}\in\OPtt\setminus{\{\mathbf{0}\}}$, $b_{\pjtt}(\mathbf{0})=0$.\par
Discretising using the collocation points $Y^{\theta}$ (\ref{eqnCollPts}) and the values
\begin{equation}
\label{eqnValb}
    b_{j}^{1,\theta}=b_{\pjtt}(\mathbf{y}^{1,\theta}),\dots,b_{j}^{\Mtt,\theta}=b_{\pjtt}(\mathbf{y}^{\Mtt,\theta})\in\bR^{-}.
\end{equation}
gives us
\begin{equation}
    \label{eqnDiscretProb}
    D_{\pjtt}v(\ymtt)\leq\bjmtt\text{ for all }m=1,\dots,\Mtt
\end{equation}
for $j=1,\dots,\cPtt$, where $v$ is the approximating function we wish to find.\par
As shown in \cite{ward_construction_nodate}, (\ref{eqnDiscretProb}) can also be expressed using functionals. We look to find a function $v\in{H}$ using the information $\bjmtt\in\bR$ generated by the functionals $\lljmtt\in{H}^{*}$, i.e. $\lljmtt(v)\leq\bjmtt$ for all $\theta=1,\dots,\Theta$, $m=1,\dots,\Mtt$ and $j=1,\dots,\cPtt$. We define these functionals as the orbital derivative (\ref{eqnDpjtt}) applied at a particular collocation point $\ymtt$, i.e.
\begin{equation}
    \label{eqnlljmtt}
    (\lljmtt)v(\cdot):=(\delta_{\ymtt}\circ{D_{\pjtt}})v(\cdot)=\nabla{v}(\ymtt)\cdot\mathbf{f}_{\pjtt}(\ymtt)
\end{equation}
for all $\theta=1,\dots,\Theta$, $m=1,\dots,\Mtt$ and $j=1,\dots,\cPtt$. We have introduced this definition of the functionals here so it is more easily comparable with that in Section \ref{secDiscret}, however the elements of the set $\Ptt$ are reordered for each collocation point in the following section when we select a basis of the functions $\mathbf{f}_{\pjtt}$ at this point, see (\ref{eqnPmtt}).\par
As in \cite{ward_construction_nodate}, the optimal reconstruction of $v$ is given by the solution of the problem
\begin{multline}
    \label{eqnMinProb}
    \min\{\|v\|_{H}:v\in{H}\text{, }\lljmtt(v)\leq\bjmtt\text{, for all }\theta=1,\dots,\Theta\text{, }\\m=1,\dots,\Mtt\text{ and }j=1,\dots,\cPtt\}.
\end{multline}
We show that this problem is equivalent to the previous one found in Section \ref{secDiscret} in the following lemma.
\begin{lemma}
    \label{lemEquivMinProbs}
    The minimisation problem (\ref{eqnMinProb}) is equivalent to the minimisation problem (\ref{eqnMinProbi}).
\end{lemma}
\begin{proof}
    Another way of writing the constraints of (\ref{eqnMinProbi}) using the definition of the functionals (\ref{eqnllil}) is
    \begin{equation}
        \label{eqnlemEquivMinProbs1}
        \nabla{v}(\xxil)\cdot\mathbf{f}_{i}(\xxil)\leq{b_{i}(\xxil)}
    \end{equation}
    for all $i=1,\dots,k$ and $\ell=1,\dots,L_{i}$. Similarly, another way of writing the constraints of (\ref{eqnMinProb}) using the definition (\ref{eqnlljmtt}) is
    \begin{equation}
        \label{eqnlemEquivMinProbs2}
        \nabla{v}(\ymtt)\cdot\mathbf{f}_{\pjtt}(\ymtt)\leq{b_{\pjtt}(\ymtt)}
    \end{equation}
    for all $\theta=1,\dots,\Theta$, $m=1,\dots,\Mtt$ and $j=1,\dots,\cPtt$.\par
    In the first part of this proof we show that satisfying the constraints (\ref{eqnlemEquivMinProbs2}) means that the constraints of (\ref{eqnlemEquivMinProbs1}) are satisfied, and in the second part we show the reverse.\medskip\\
    (\ref{eqnlemEquivMinProbs2}) $\Rightarrow$ (\ref{eqnlemEquivMinProbs1}): Fix $i\in{I}$ and $\ell\in\{1,\dots,\Li\}$. This gives us a specific collocation point $\xxil$. Let $P^{\eta_{r}}\subseteq{I}$ be the sets such that $i$ is an element for $1\leq{\eta_{r}}\leq\Theta$, $r=1,\dots,R$ and no other set $\Ptt$ for $1\leq\theta\leq\Theta$. This means that (\ref{eqnXCollPtsUnionYCollPts}) holds and we have that
    \begin{equation}
        \label{eqnlemEquivMinProbs3}
        \xxil\in\bigcup_{r=1,\dots,R}Y^{\eta_{r}}.\nonumber
    \end{equation}
    Therefore, there exists a $\theta=\eta_{r}$ for some $r=1,\dots,R$ and a $m=1,\dots,\Mtt$ such that $\xxil=\ymtt$. As we have that $i\in\Ptt$, there exists a $j=1,\dots,\cPtt$ such that $i=\pjtt$. Now using this and $\xxil=\ymtt$ in (\ref{eqnlemEquivMinProbs2}) we have proven this implication.\medskip\\
    (\ref{eqnlemEquivMinProbs1}) $\Rightarrow$ (\ref{eqnlemEquivMinProbs2}): Fix $\theta=1,\dots,\Theta$, $m=1,\dots,\Mtt$ and $j=1,\dots,\cPtt$. Set $i=\pjtt$, then $i\in\Ptt$. By the definition of $\OPtt$ (\ref{eqnOP}) as we have that $i\in\Ptt$, we have that $\OPtt\subseteq{C_{i}}$. Therefore, using the definitions of $X_{i}$ (\ref{eqnCollPtsi}) and $Y^{\theta}$ (\ref{eqnCollPts}) we have for $i\in\Ptt$ that $Y^{\theta}\subseteq{X_{i}}$. This means that for this $i$ there exists an $\ell=1,\dots,\Li$ such that $\ymtt=\xxil$. Using this and that $\pjtt=i$ in (\ref{eqnlemEquivMinProbs1}) we have completed the proof.
\end{proof}

\subsection{Selection of basis}
\label{secBasis}

If the functionals are linearly independent, we can use techniques from \cite{wendland2004scattered} to find a solution to the minimisation problem (\ref{eqnMinProb}). However, this is not necessarily the case and we must use techniques previously described in \cite{ward_construction_nodate} to deal with this. We fix a point $\ymtt\in\OPtt$, consider whether $\mathbf{f}_{\pjtt}(\ymtt)$ for $j=1,\dots,\cPtt$ are linearly independent, and if not, select a basis of the space spanned by these vectors.\par 
To this end, fix $\ymtt\in\OPtt$ and define the set 
\begin{equation}
    \label{eqnRmtt}
    \Rmtt:=\{r_{1}^{m,\theta},\dots,r_{\rrmtt}^{m,\theta}\}\subseteq\Ptt
\end{equation}
to be a set of indices such that the vectors $\mathbf{f}_{\rjmtt}(\ymtt)$ with $j=1,\dots,\rrmtt$ are a basis of 
\begin{equation}
    \label{eqnSpanBasis}
    \text{span(}(\mathbf{f}_{\pjtt}(\ymtt))_{\pjtt\in\Ptt})\nonumber
\end{equation}
with $\rrmtt$ being its rank. The elements of $\Ptt\setminus{\Rmtt}$ are labelled $\pi_{s}^{m,\theta}$ where $s=1,\dots,\cPtt-\rrmtt$. We allow the case where the vectors $\mathbf{f}_{\pjtt}(\ymtt)$ are all linearly independent, this would mean $\Rmtt=\Ptt$.\par
Therefore, we consider the vector
\begin{equation}
    \label{eqnPmtt}
    \Pmtt:=\left(r_{1}^{m,\theta},\dots,r_{\rrmtt}^{m,\theta},\pi_{1}^{m,\theta},\dots,\pi_{\cPtt-\rrmtt}^{m,\theta}\right)
\end{equation}
and write its components as $\pjmtt:=(\Pmtt)_{j}$ for $j=1,\dots,\cPtt$. This means we fix an ordering of the set $\Ptt$ for the collocation point $\ymtt$ where the indices that represent the basis are at the beginning of the vector. We find the vector (\ref{eqnPmtt}) for each point $\ymtt$ for $\theta=1,\dots,\Theta$ and $m=1,\dots,\Mtt$.\par
From this point onward when we refer to the functionals $\lljmtt$ we mean the definition (\ref{eqnlljmtt}) but replacing $\pjtt$ with the components $\pjmtt$ from the vector $\Pmtt$ (\ref{eqnPmtt}).\par
For notational convenience we now define
\begin{equation}
    \label{eqnRhocP}
    \rho^{\theta}:=\sum_{m=1}^{\Mtt}\rrmtt,\phantom{x}\rho:=\sum_{\theta=1}^{\Theta}\sum_{m=1}^{\Mtt}\rrmtt\text{ and }\cP:=\sum_{\theta=1}^{\Theta}\sum_{m=1}^{\Mtt}\cPtt.
\end{equation}\par
From (\ref{eqnPmtt}) we have that the functionals $\lljmtt$ for $j=1,\dots,\rrmtt$ are a basis for the whole set of functionals with $j=1,\dots,\cPtt$. Therefore, the functionals at the point $\ymtt$ are expressed as a linear combination of this basis and we write this as
\begin{equation}
    \label{eqnllssmtt}
    \llssmtt=\sum_{j=1}^{\rrmtt}\aassjmtt\lljmtt
\end{equation}
for all $\sigma=1,\dots,\cPtt$, for some appropriate $\aassjmtt\in\bR$. We assemble the coefficients $\aassjmtt$ for all of the collocation points $\ymtt$ 
as elements of a matrix $Q\in\bR^{\cP\times\rho}$ such that
\begin{equation}
    \label{eqnQ}
    Q:=\left(\begin{array}{c}
    I\\
    \Tilde{Q}
    \end{array}\right)
\end{equation}
where $I$ is the $\rho\times\rho$ identity matrix accounting for the coefficients $\alpha_{s,j}^{m,\theta}$ for $\theta=1,\dots,\Theta$, $m=1,\dots,\Mtt$, $s,j=1,\dots,\rrmtt$ and $\Tilde{Q}$ is a block diagonal matrix that contains the remaining coefficients. The structure of this matrix is described in \cite{ward_construction_nodate}.

\subsection{Quadratic programming problem}
\label{secQuadProg}

We now wish to show that the solution of (\ref{eqnMinProb}) with the functionals $\lljmtt$ defined above is of the form
\begin{equation}
    \label{eqnLyapv}
    v(\mathbf{x})=\sum_{\theta=1}^{\Theta}\sum_{m=1}^{\Mtt}\sum_{j=1}^{\rrmtt}\bbjmtt(\lljmtt)^{\mathbf{y}}\Phi(\mathbf{x},\mathbf{y})
\end{equation}
where the coefficients $\bbjmtt$ are elements of the vector $\boldsymbol{\beta}\in\bR^{\rho}$ which satisfies
\begin{equation}
    \label{eqnbTAb}
    \begin{cases}
        \text{\rm{minimise }}\boldsymbol{\beta}^{T}A\boldsymbol{\beta}\\
        \text{\rm{subject to }}QA\boldsymbol{\beta}\leq\mathbf{b}
    \end{cases}
\end{equation}
To find the matrices $Q$ and $A$ we applied our functionals to our proposed solution (\ref{eqnLyapv}) and used the inequalities in (\ref{eqnMinProb}).\par
The elements of $A\in\bR^{\rho\times\rho}$ are
\begin{equation}
    \label{eqnAElem}
    \aijeettlm:=(\delta_{\ymtt}\circ{D_{\pjmtt}})^{\mathbf{x}}(\delta_{\ylee}\circ{D_{\pilee}})^{\mathbf{y}}\Phi(\mathbf{x},\mathbf{y})=\langle(\lljmtt)^{\mathbf{x}}\Phi(\cdot,\mathbf{x}),(\llilee)^{\mathbf{y}}\Phi(\cdot,\mathbf{y})\rangle_{H}
\end{equation}
for $\eta,\theta=1,\dots,\Theta$, $\ell=1,\dots,\Mee$, $m=1,\dots,\Mtt$, $i=1,\dots,\rrlee$ and $j=1,\dots,\rrmtt$. This matrix is symmetric and accounts for the linearly independent functionals. Applying the matrix $Q$ (\ref{eqnQ}) accounts for the linearly dependent functionals as a linear combination of the independent ones (\ref{eqnllssmtt}). The values of the vector $\mathbf{b}\in\bR^{\cP}$ are the elements $\bjmtt$ from the minimisation problem (\ref{eqnMinProb}) for all $\theta=1,\dots,\Theta$, $m=1,\dots,\Mtt$ and $j=1,\dots,\cPtt$.\par
We can also see that minimising $\boldsymbol{\beta}^{T}A\boldsymbol{\beta}$ as in (\ref{eqnbTAb}) is the same as minimising $\|v\|_{H}$ in the following way
\begin{eqnarray}
    \|v\|^{2}_{H} & = & \left\langle\sum_{\eta=1}^{\Theta}\sum_{\ell=1}^{\Mee}\sum_{i=1}^{\rrlee}\bbilee(\llilee)^{\mathbf{y}}\Phi(\cdot,\mathbf{y}),\sum_{\theta=1}^{\Theta}\sum_{m=1}^{\Mtt}\sum_{j=1}^{\rrmtt}\bbjmtt(\lljmtt)^{\mathbf{z}}\Phi(\cdot,\mathbf{z})\right\rangle_{H}\nonumber\\
    & = & \sum_{\eta=1}^{\Theta}\sum_{\ell=1}^{\Mee}\sum_{i=1}^{\rrlee}\sum_{\theta=1}^{\Theta}\sum_{m=1}^{\Mtt}\sum_{j=1}^{\rrmtt}\bbilee\bbjmtt\langle(\llilee)^{\mathbf{y}}\Phi(\cdot,\mathbf{y}),(\lljmtt)^{\mathbf{z}}\Phi(\cdot,\mathbf{z})\rangle_{H}\nonumber\\
    & = & \sum_{\eta=1}^{\Theta}\sum_{\ell=1}^{\Mee}\sum_{i=1}^{\rrlee}\sum_{\theta=1}^{\Theta}\sum_{m=1}^{\Mtt}\sum_{j=1}^{\rrmtt}\bbilee\bbjmtt\aijeettlm\nonumber\\
    & = & \boldsymbol{\beta}^{T}A\boldsymbol{\beta}.\nonumber
\end{eqnarray}\par
Now we are able to show that if there is a feasible solution to (\ref{eqnbTAb}), the solution of (\ref{eqnMinProb}) is of the form (\ref{eqnLyapv}). We do this with the following lemma and theorem, the proofs of which are found in \cite{ward_construction_nodate}. 
\begin{lemma}
    \label{lemUniqueSolbTAb}
    If there is a feasible solution (i.e. a $\boldsymbol{\beta}$ that satisfies the constraints) of the problem (\ref{eqnbTAb}), then the solution of the minimising problem (\ref{eqnbTAb}) is unique.
\end{lemma}

\begin{theorem}
    \label{thmSolMinProb}
    If there is a feasible solution to the problem (\ref{eqnbTAb}), the minimiser of the problem (\ref{eqnMinProb}) is of the form (\ref{eqnLyapv}) where the coefficients $\bbjmtt$ are determined by the solution of the problem (\ref{eqnbTAb}).
\end{theorem}
\begin{rem}
    \label{remFeasible}
    We can check if there is a feasible solution to (\ref{eqnbTAb}) computationally; then there exists a unique minimiser to each discretised problem (\ref{eqnthmConvergence4}). This is due to Theorem \ref{thmSolMinProb} as we know that there will exist a Lyapunov function for the system of the form (\ref{eqnLyapv}). We choose these functions as our sequence $v_{n}^{\varepsilon}$ in Theorem \ref{thmConvergence} with $\varepsilon=0$.
\end{rem}

\section{Examples}
\label{secExamples}

In this section we give two examples where we construct Lyapunov functions for switched systems using Algorithm \ref{algMain}. Example \ref{exNLArb} is a non-linear arbitrary switched system and Example \ref{exLArb+VarStructR3} uses a combination of arbitrary and variable structure switching in $\bR^{3}$. In both of these examples we have taken the Lyapunov function $v$ that is constructed by our algorithm, and computed a new Lyapunov function $v_{\rm new}(\mathbf{x}):=v(\mathbf{x})-v(\mathbf{0})$. This function has its minimum value at zero ensuring we meet the assumption outlined in (\ref{eqnLyapSwitch}). We are able to do this as the orbital derivatives of $v(\mathbf{x})$ and $v_{\rm new}(\mathbf{x})$ are the same.\par
The Wendland function that we use in the implementation of Algorithm \ref{algMain} is $\phi_{4,2}$ from \cite{giesl2007constructionRB} which means that we have
\begin{equation}
    \label{eqnWendlandFunc}
    \Psi_{0}(r)=(1-cr)^{6}_{+}[35(cr)^{2}+18cr+3]
\end{equation}
where $c=\frac{5}{6}$.

\begin{example}{\rm (Non-linear arbitrary switched system)}
\label{exNLArb}
    $\Omega=[-0.5,0.5]^{2}$ and the values of the vector $\mathbf{b}$ are found using that $b_{\pjtt}(x,y)=-0.25x^{2}-0.25y^{2}$ for all $\theta=1,\dots,\Theta$ and $j=1,\dots,\cPtt$. Consider the following autonomous systems with asymptotically stable equilibrium points at the origin
    \begin{equation}
        \label{eqnexNLArb1}
        \dot{\mathbf{x}}=\mathbf{f}_{1}(x,y)\text{, where }\mathbf{f}_{1}(\mathbf{x}):=\left(\begin{array}{c}
            y\\
            x-2\tan^{-1}(x+y)
        \end{array}\right)
    \end{equation}
    and
    \begin{equation}
        \label{eqnexNLArb2}
        \dot{\mathbf{x}}=\mathbf{f}_{2}(x,y)\text{, where }\mathbf{f}_{2}(\mathbf{x}):=\left(\begin{array}{c}
            y\\
            -x+\frac{1}{3}x^{3}-y
        \end{array}\right).
    \end{equation}
    The system (\ref{eqnexNLArb1}) is from Exercise 2.4 in \cite{KhalilNonlinear}, and (\ref{eqnexNLArb2}) is from Example 1 in \cite{giesl2015computation}. In this example we switch arbitrarily between these two systems over the whole space $\Omega$. This is written as
    \begin{equation}
        \label{eqnexNLArb3}
        \dot{\mathbf{x}}=\mathbf{f}_{p}(\mathbf{x})\text{, }p\in\{1,2\}
    \end{equation}
    where $\mathbf{x}\in\Omega$. We apply our algorithm using $\Theta=1$ and $\Omega^{P}=\Omega$, where $P=\{1,2\}$. The collocation points are selected to be $\frac{1}{6}\bZ^{2}\cap\Omega\setminus\{(0,0)\}$, this gives us 48 collocation points in total. The Lyapunov function produced for this system can be seen in Figure \ref{fig:LyapexNLArb}, and its orbital derivatives in Figure \ref{fig:OrbDervexNLArb} which are negative except at the origin.
    \begin{figure}[h]
        \centering
        \includegraphics[width=0.5\linewidth]{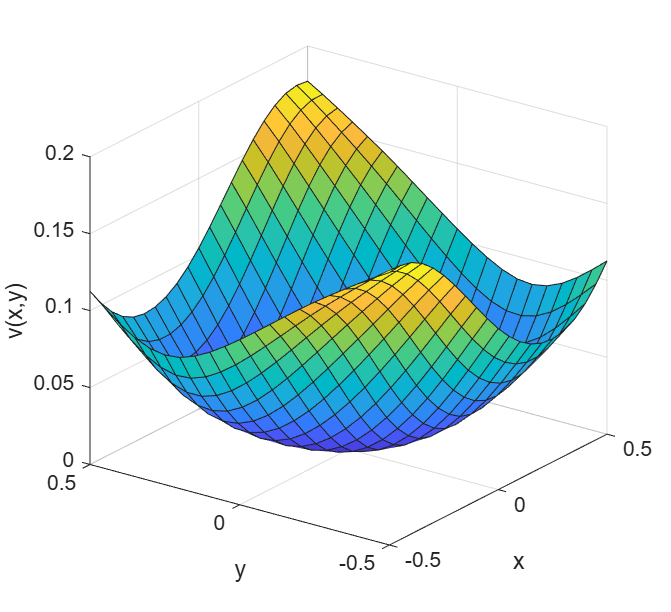}
        \caption{Lyapunov function constructed for the arbitrary switched system (\ref{eqnexNLArb3}) from Example \ref{exNLArb}.}
        \label{fig:LyapexNLArb}
    \end{figure}
    \begin{figure}[h]
        \centering
        \includegraphics[width=0.45\linewidth]{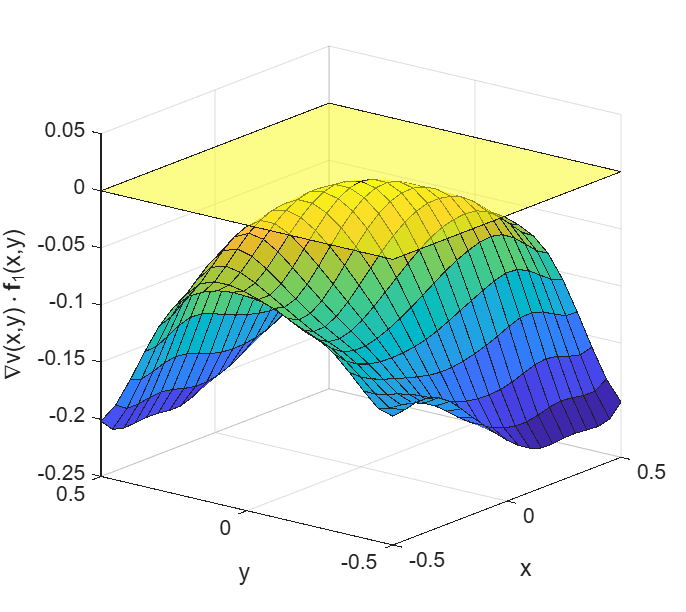}
        \includegraphics[width=0.45\linewidth]{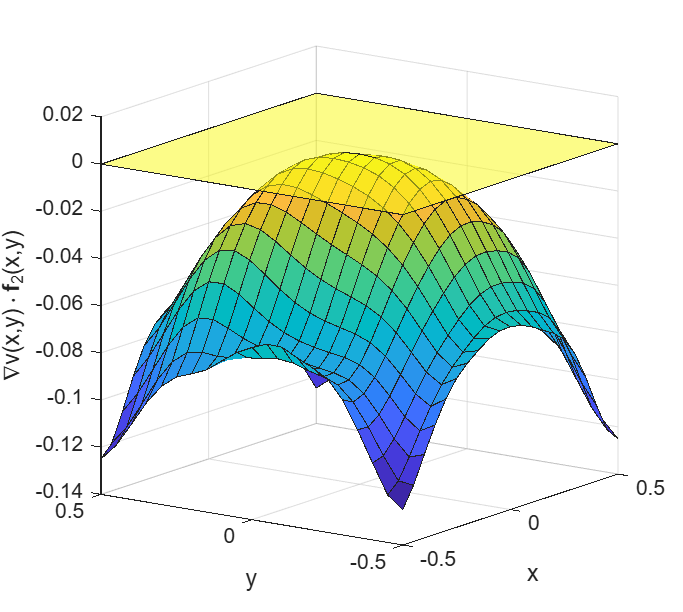}
        \caption{Plots of the orbital derivatives of the Lyapunov function from Figure \ref{fig:LyapexNLArb}: $\nabla{v}(x,y)\cdot\mathbf{f}_{1}(x,y)$ (left) and $\nabla{v}(x,y)\cdot\mathbf{f}_{2}(x,y)$ (right). In both figures the plane $\{z=0\}$ has been plotted to illustrate that the orbital derivatives are negative except for at the origin.}
        \label{fig:OrbDervexNLArb}
    \end{figure}
\end{example}

\begin{example}{\rm (Non-linear arbitrary and variable structure system in $\bR^3$)}
\label{exLArb+VarStructR3}
    $\Omega=[-0.8,0.8]^{3}$  and the values of the vector $\mathbf{b}$ are found using that $b_{\pjtt}(x,y,z)=-0.25x^{2}-0.25y^{2}-0.25z^{2}$ for all $\theta=1,\dots,\Theta$ and $j=1,\dots,\cPtt$. Take the following two systems with asymptotically stable equilibrium points at the origin
    \begin{equation}
        \label{eqnexLArb+VarStructR31}
        \dot{\mathbf{x}}=\mathbf{f}_{1}(x,y,z)\text{, where }\mathbf{f}_{1}(x,y,z):=\left(\begin{array}{c}
            x(x^{2}+y^{2}-1)-y(z^{2}+1) \\
            y(x^{2}+y^{2}-1)+x(z^{2}+1) \\
            10z(z^{2}-1)
        \end{array}\right)
    \end{equation}
    and
    \begin{equation}
        \label{eqnexLArb+VarStructR32}
        \dot{\mathbf{x}}=\mathbf{f}_{2}(x,y,z)\text{, where }\mathbf{f}_{2}(x,y,z):=\left(\begin{array}{c}
            -2x+x^{3} \\
            -y+x^{2} \\
            -z
        \end{array}\right).
    \end{equation}
    The system (\ref{eqnexLArb+VarStructR31}) is from Example 2 of \cite{giesl2015computation}, and (\ref{eqnexLArb+VarStructR32}) is from Exercise 4.32 of \cite{KhalilNonlinear}. We define $\mathbf{f}_{1}$ on the whole space $\Omega$, and $\mathbf{f}_{2}$ on a subset  $S$ of $\Omega$ defined as
    \begin{equation}
        \label{eqnexLArb+VarStructR33}
        S:=\{\mathbf{x}\in\Omega:z\geq{10x^{2}+10y^{2}}\}.
    \end{equation}
    The set $S$ is shown on the left of Figure \ref{fig:LevelSetexLArb+VarStructR3} to show how the space $\Omega$ is partitioned.\par
    This means we can define a switched system with both time- and state-dependent switching as follows
    \begin{equation}
        \label{eqnexLArb+VarStructR34}
        \dot{\mathbf{x}}=\mathbf{f}_{p}(\mathbf{x})\text{, }p\in\{1,2\}
    \end{equation}
    where $p=1$ when $\mathbf{x}\in\Omega\setminus{S}$, and $p$ switches arbitrarily between 1 and 2 when $\mathbf{x}\in{S}$. We can now apply Algorithm \ref{algMain} by taking $\Theta=2$; $P^{1}=\{1\}$ and $P^{2}=\{1,2\}$; and $\Omega^{P^{1}}=\Omega\setminus{S}$ and $\Omega^{P^{2}}=S$. The collocation points are selected to be $\frac{1}{6}\bZ^{3}\cap\Omega\setminus\{(0,0,0)\}$, this gives us 324 collocation points in total. Figure \ref{fig:LevelSetexLArb+VarStructR3} shows the partition of $\Omega$ on the left, with $\Omega^{P^{2}}$ in red and $\Omega^{P^{1}}$ making up the remaining space. On the right, the level set $S_{0.085}$ of the Lyapunov function has been plotted. The orbital derivatives of the Lyapunov function were also found and are negative in $\Omega$ except for at the origin.
    \begin{figure}
        \centering
        \includegraphics[width=0.45\linewidth]{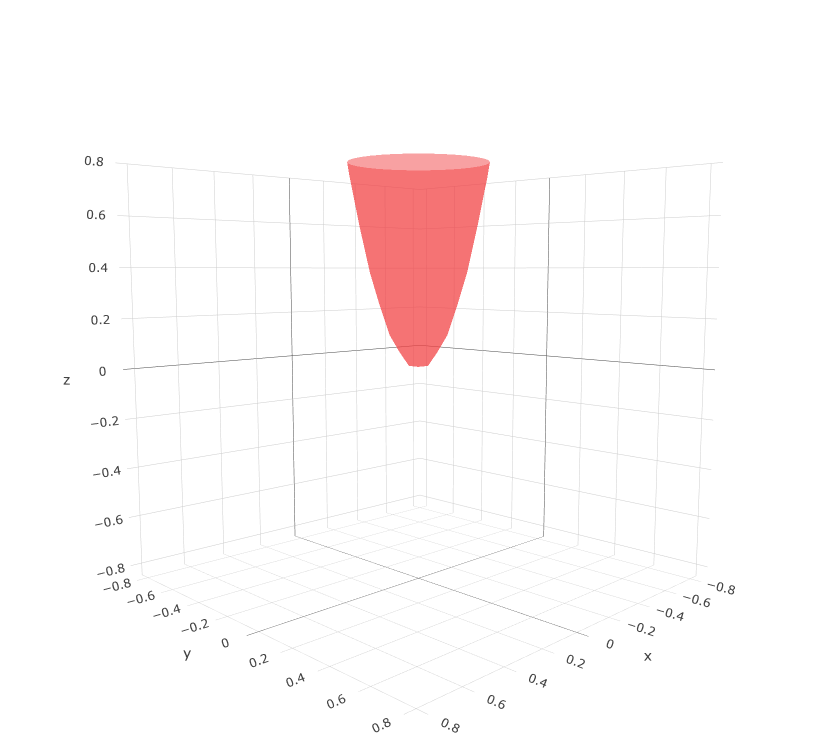}
        \includegraphics[width=0.45\linewidth]{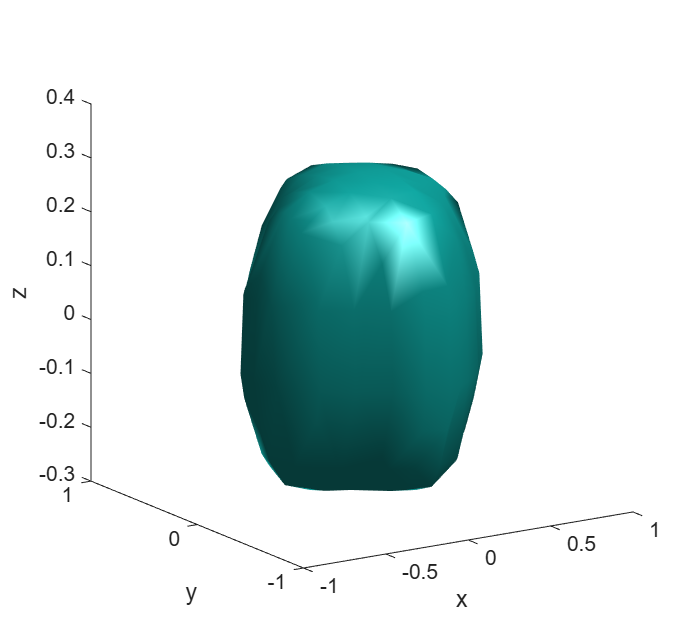}
        \caption{Left: The space $\Omega$ partitioned into two sets, $\Omega^{P^{2}}=S$ (\ref{eqnexLArb+VarStructR33}) in red and $\Omega^{P^{1}}$ being the remaining space.\\
        Right: The level set $S_{0.085}$ (see (\ref{eqnthmStability1})) of the Lyapunov function constructed for the switched system (\ref{eqnexLArb+VarStructR34}) from Example \ref{exLArb+VarStructR3}. The orbital derivatives of the Lyapunov function were found and are negative in $\Omega$ except for at the origin.}
        \label{fig:LevelSetexLArb+VarStructR3}
    \end{figure}
\end{example}

\begin{rem}
    \label{remVerif}
    Checking that the orbital derivative is negative was done by computing the values at the evaluation points. To obtain rigorous guarantees for the negativity you can either: use the form of $v$ obtained by meshfree collocation and Taylor-type estimates, similar to \cite{giesl2015computation}, or  interpolate the function $v$ by a continuous piecewise affine (CPA) function, for which estimates can easily be checked, similar to \cite{giesl2019verification}.
\end{rem}


\section{Conclusion}
\label{secConc}


In this paper, we have shown in Theorem \ref{thmConvergence} that the algorithm to construct a Lyapunov function for the switched system (\ref{eqnfis}) originally presented in \cite{ward_construction_nodate} converges under some specific assumptions. One such assumption is that the fill distance (\ref{eqnthmConvergence3}) decreases as the number of collocation points increases. In Section \ref{secFillDistance} we gave a way to choose collocation points so that the assumption on the fill distance is met. The construction method from \cite{ward_construction_nodate} applies to both arbitrary switched and variable structure systems, as well as systems that have a combination of time- and state-dependent switching. We have provided further details on this method and given two new examples. Another main result is that the existence of a Lyapunov function for a switched system implies that the origin is a uniformly asymptotically stable equilibrium point of the system.\par
In Theorem \ref{thmConvergence} we also assumed the existence of a function $V_{0}$. In future work, we will provide sufficient conditions for the existence of such a function.

\end{document}